\documentclass[11pt]{amsart}

\usepackage{amsthm}
\usepackage{amsmath}
\usepackage{amssymb}
\usepackage{color,xcolor}
\usepackage{ lscape}
\usepackage{enumitem} 
\usepackage{placeins, afterpage}
\usepackage{array}
\usepackage{caption}

\usepackage{floatrow}
\DeclareFloatFont{tiny}{\tiny}% "scriptsize" is defined by floatrow, "tiny" not
\newcolumntype{P}[1]{>{\centering\arraybackslash}p{#1}}

\newcommand{\tworows}[2]{\begin{array}{l}
\scriptstyle  #1  \\ [2pt]
\scriptstyle  #2  \end{array}}

\newcommand{\SPAN}{\mathop\mathrm{span}}
\newcommand{\TR}{\mathop\mathrm{tr}}

\newcommand{\ANN}{\mathop\mathrm{Ann_{\mathbb G} (\mathbb H)}}
\newcommand{\AN}{\mathop\mathrm{Ann_{\mathbb G}}}

\newcommand{\RES}[1]{R^{(#1)}_p}

\newtheorem{thm}{Theorem}[section]
\newtheorem{theorem}[thm]{Theorem}

\newtheorem{cor}[thm]{Corollary}

\newtheorem{lemma}[thm]{Lemma}

\newtheorem{prop}[thm]{Proposition}
\newtheorem{defn}[thm]{Definition}

\theoremstyle{remark}
\newtheorem{remark}[thm]{Remark}

\newtheorem{rem}[thm]{Remark}
\newtheorem{ex}[thm]{Example}

\newcommand{\jh}[1]{{{#1}}}

\begin{document}

\title{Toward the Classification of Biangular Harmonic Frames}
\author[ Casazza, Farzannia, Haas, Tran
 ]{Peter G. Casazza, Amineh Farzannia, John I. Haas, and Tin T. Tran}
\address{Department of Mathematics, University
of Missouri, Columbia, MO 65211-4100}

\thanks{The authors were supported by
 NSF DMS 1609760; NSF ATD 1321779, ARO W911NF-16-1-0008}

\email{casazzap@missouri.edu}
\email{afwwc@mail.missouri.edu }
\email{terraformthedreamscape@gmail.com}
\email{tttrz9@mail.missouri.edu}


\begin{abstract}
Equiangular tight frames (ETFs) and biangular tight frames (BTFs) - sets of unit vectors with basis-like properties whose pairwise absolute inner products admit exactly one or two values, respectively - 
are useful for many applications.  A well-understood class of ETFs are those which manifest as harmonic frames - vector sets defined in terms of the characters of finite abelian groups - because they are characterized by combinatorial objects called difference sets.  

This work is dedicated to the study of the underlying combinatorial structures of harmonic BTFs.  We show that if a harmonic frame is generated by a divisible difference set, a partial difference set or by a special structure with certain Gauss summing properties  - all three of which are generalizations of difference sets that fall under the umbrella term ``bidifference set" -  then it is either a BTF or an ETF.  However, we also show that the relationship between harmonic BTFs and bidifference sets is not as straightforward as the correspondence between harmonic ETFs and difference sets, as there are examples of bidifference sets that do not generate harmonic BTFs.  In addition, we study another class of combinatorial structures, the nested divisible difference sets, 
which yields an example of a harmonic BTF that is not generated by a bidifference set.
\end{abstract}

%\begin{keyword}
%equiangular tight frames; biangular tight frames; harmonic frames; difference sets; divisible difference sets; relative difference sets; partial difference sets; Grassmannian line packings
%\end{keyword}

\maketitle

\section{Introduction}
The Grassmannian line packing problem has received attention from mathematicians for more than half a century \cite{Haantjes1948, LemmensSeidel1973, Koornwinder1976}.  Roughly speaking, 
% and a corresponding set of represntative unit vectors,
 the task is to  arrange a fixed number of lines in Euclidean space so that they are as ``spread out as possible," a phrase that is assigned  formal meaning according to a number of different but related definitions \cite{ConwayHardinSloane1996, MR2915706}, typically involving the minimization of some function of the system's {\it angle set} - the set of pairwise absolute inner products between the lines' representative unit vectors. 
%Beside its appeal as the objects of a long-standing academic challenge, 
Such optimal arrangements of lines are useful for a wide range of disciplines, including areas of science and engineering, such as
 quantum state tomography \cite{Zauner1999, MR2059685, 1523643, RoyScott2007} and wireless communications \cite{MR1984549, MR2028016, MR2021601}, and pure mathematical subjects, like graph theory \cite{Seidel1976, MR2021601, 2016arXiv160203490F, 2014arXiv1408.0334H}.  

Systems of {\it equiangular} lines - line sets with angle sets of cardinality one - 
are perhaps the most well-studied systems \cite{SustikTroppDhillonHeath2007, MR2890902, ScottGrassl2010, Fickus:2015aa, MR2711357, MR2460526, MR3516710, MR1455862} because, according to a lower bound of Welch \cite{Welch1974},  they can form optimal packings.  However, it is well-known that when the number of lines is too large relative to the dimension of the ambient vector space, then they cannot be equiangular \cite{LemmensSeidel1973, Koornwinder1976}; furthermore, there are cases where the number does not exceed this threshold for which equiangular configurations are not possible \cite{Fickus:2015aa, Szollosi2014b}.

In light of this appealing qualitative description, it is natural to hope, in settings where equiangular configurations are not possible, that the cardinalities of the angle sets of optimal line packings might still satisfy some sort of minimality condition.
For this reason,  we consider systems of {\it biangular} lines - sets of lines with angles sets of cardinality two -  which have been studied previously in \cite{DelsarteGoethalsSeidel1975, Hoggar1982, MR3047911, Neumaier1989, MR3325226, MR1455862, bgb15}.  For example, the authors of~\cite{bgb15} prove that several instances of biangular line sets satisfy certain energy minimization properties.  In this work, we focus on biangular line sets generated by {\it harmonic frames}.

A {\it harmonic frame} is a set of unit vectors whose entries are determined according to some subselection from the characters of a finite abelian group, and we say that it is equiangular or biangular, respectively, if the lines that the vectors generate satisfy the corresponding property.  In recent years, harmonic frames have been used to construct various types of optimal line packings \cite{MR1984549, XiaZhouGiannakis2005, MR2446568, GodsilRoy2009, MR3557826}.  Equiangular harmonic frames are well-understood, as it is known that a harmonic frame is equiangular if and only if the underlying subselection of characters corresponds to a  {\it difference set}  \cite{MR1984549, XiaZhouGiannakis2005,  MR2446568}, an object that has received considerable attention within the combinatorial literature throughout most of the last century \cite{MR0001221, MR2246267, MR1440858}.  
In certain settings where equiangular lines are not possible,  
biangular harmonic frames have also been used to construct optimal Grassmannian packings, for example lines generated by maximal sets of mutually unbiased bases \cite{GodsilRoy2009} or the picket fence constructions of \cite{MR3557826}. 

In this work, 
we
study and classify biangular harmonic frames in terms of their underlying combinatorial structures. 
Unlike equiangular harmonic frames, which admit a simple characterization in terms of difference sets, we find that the biangular case is not so straightforward.  Using an approach based on the Fourier transform, we show that a harmonic frame is biangular if its underlying subselection of characters corresponds to either a {\it partial difference set} or a {\it divisible difference set} - both well-studied generalizations of difference sets  \cite{MR1440858, MR1277942}.  Motivated by this, we define {\it bidifference sets} - a more general combinatorial structure that includes  partial difference sets, divisible difference sets and a third class that we show to generate biangular harmonic frames, {\it the Gaussian difference sets}.  Given these results, it seems natural to expect that bidifference sets might be the ``right" notion with which to characterize biangular harmonic frames; however, we provide an example of a bidifference set which does not generate a biangular harmonic frame.  Furthermore, we study a class of combinatorial structures which includes the divisible difference sets but is not contained in the bidifference sets, the {\it nested divisible difference sets}, which admits an example of a biangular harmonic frame that is not generated by a bidifference set.

The remainder of this article is outlined as follows.  In Section~\ref{sec_pre}, we fix notation and recall some basic facts from frame theory and character theory.  In Section~\ref{sec_mod}, we develop the theory of {\it modulation operators}, the Fourier transform-based tool that we use to analyze harmonic frames in the following section.  In Section~\ref{sec_btfs}, we  
study the relationship between biangular harmonic frames and their underlying combinatorial structures, as described in the preceding paragraph.  Finally, in the appendix, we
% summarize the main results of this work and
 tabulate  several examples of infinite families of biangular harmonic frames generated by divisible difference sets and partial difference sets, including detailed information about the corresponding angle sets.

%In etween biangular lines generated by harmonic frames 
%
%
%construct other types optimal line packings \cite{}.  By selecting the underlying characters in accordance with a generalized notion of difference sets called {\it relative difference sets} (see Definition~\ref{}), one can produce a set of biangular lines which can then be 

\section{Preliminaries}\label{sec_pre}
\subsection{Frame Theory}\label{sec:prelim}
Let $\{e_j\}_{j=1}^m$ denote the canonical orthonormal basis for the finite dimensional Hilbert space $\mathbb F^m$, where $\mathbb F = \mathbb R$ or $\mathbb C$, and let $I_m$ denote the $m \times m$ identity matrix. A set of vectors $\mathcal F = \{f_j\}_{j =1}^n \subset \mathbb F^m$ is a {\bf (finite) frame} if $\SPAN \{f_j\}_{j=1}^n = \mathbb F^m.$

A frame $\mathcal F = \{f_j\}_{j=1}^n$  is {\bf $a$-tight} if $\sum_{j=1}^n f_j \otimes f_j^*= a I_m$ for some $a>0$, where  $f_j \otimes f_j^*$ denotes the orthogonal projection onto the $1$-dimensional subspace spanned by $f_j$, and  $\mathcal F$ is {\bf unit-norm} if each frame vector has norm $\|f_j\|=1$.  If $\mathcal F$ is unit-norm and $a$-tight, then $a=\frac{n}{m}$ because
$$n= \sum_{j'=1}^m\sum_{j=1}^n |\langle e_{j'}, f_j \rangle|^2 = \sum_{j'=1}^m\sum_{j=1}^n \TR(f_j \otimes f_j^* e_{j'} \otimes e_{j'}^*) = a \sum_{j'=1}^m \| e_{j'} \|^2 = am,$$
which also implies that every such frame satisfies the identity 
\begin{equation}\label{eq_unt_cond}
\sum\limits_{j'=1}^n |\langle f_j, f_{j'} \rangle|^2 = \frac{n}{m} \text{ for every } j \in \{1,...,n\}.
\end{equation}

Given any unit-norm frame $\mathcal F = \{ f_j \}_{j=1}^n$, its {\bf frame angles} are the elements of the set 
$
\Theta_{\mathcal F} : =\left\{ |\langle f_j, f_{j'} \rangle | : j \neq j' \right\},
$
% i.e. \textem{the angle set of $\Phi$}
 and we say that $\mathcal F$ is {\bf $d$-angular} if $|\Theta_{\mathcal F}| =d$ for some $d \in \mathbb N$.  
In the special case that $\mathcal F$ is $1$-angular or $2$-angular, then we say that it is {\bf equiangular} or {\bf biangular}, respectively.  

If $\mathcal F = \{ f_j \}_{j=1}^n$ is $d$-angular with frame angles $\alpha_1, \alpha_2,...,\alpha_d$, then we say that $\mathcal F$ is {\bf equidistributed} if there exist $\tau_1, \tau_2, ..., \tau_d \in \mathbb N$ such that 
$$
\left| \left\{ j' \in \{1,...,n \} : j' \neq j, |\langle f_j, f_{j'} \rangle | = \alpha_l \right\} \right| = \tau_l
$$
for every $j \in \{1,2,...,n\}$ and every $l \in \{1,2,...,d\}$.  In this case, we call the positive integers $\tau_1, \tau_2,...,\tau_d$  the {\bf frame angle multiplicities} of $\mathcal F$ and remark that $\sum_{j=1}^d \tau_j= n-1$.

According to the lower bound of Welch \cite{Welch1974}, if $\mathcal F=\{f_j\}_{j=1}^n$ is a unit-norm frame for $\mathbb F^m$, then the maximal magnitude among all inner products between distinct frame vectors is bounded below by
\begin{equation*}
\max\limits_{j \neq j'} |\langle f_j, f_{j'} \rangle| \geq \sqrt{\frac{n-m}{m(n-1)}},
\end{equation*}
and it is well-known \cite{ Fickus:2015aa} that $\mathcal F$ achieves this bound if and only if $\mathcal F$ is an equiangular, tight frame ({\bf ETF}). 

In general, unit-norm, tight, $d$-angular frames are not equidistributed; for instance, see the constructions in \cite{MR3557826}.  However, the equidistributed property holds for the special case that $d \leq 2$.  If $\mathcal F = \{ f_j\}_{j=1}^n$ is an ETF, then $\mathcal F$ is clearly equidistributed with the single frame angle $\alpha_1=\sqrt{\frac{n-m}{m(n-1)}}$ and corresponding frame angle multiplicity $\tau_1=n-1$.    If $\mathcal F = \{ f_j\}_{j=1}^n$ is a biangular tight frame ({\bf BTF}), then after substituting the frame angles, $\alpha_1$ and $\alpha_2$, into Equation~(\ref{eq_unt_cond}), it is evident that the distribution of these two values as summands in the equation's right-hand side must be the same for any choice of $j\in\{1,...,n\}$;   in other words, for Equation~(\ref{eq_unt_cond})
 to hold for every $j\in\{1,...,n\}$, there must exist positive integers $\tau_1$ and $\tau_2$ such that 
\begin{equation}\label{eq_btf_cond}
\frac{n-m}{m} =\sum\limits_{\tiny \begin{array}{cc} j'=1 \\ j' \neq j \end{array}}^n | \langle f_j, f_{j'} \rangle |^2 = \tau_1 \alpha_1^2 + \tau_2 \alpha_2^2,
\end{equation}
which means that $\mathcal F$ is equidistributed (see Proposition~5.1 of~\cite{2016arXiv160502012C} for an alternative proof of this claim).  \jh{Furthermore, solving Equation~(\ref{eq_btf_cond}) in conjunction with the identity $\tau_1 + \tau_2 = n-1$ yields the values of the frame angle multiplicities of BTFs, which we record in the following proposition.
\begin{prop}\label{prop_btf_multiplicities}
If $\mathcal F = \{f_j\}_{j=1}^n$ is a biangular tight frame for $\mathbb F^m$ with frame angles $\alpha_1$ and $\alpha_2$, then the correponding frame angle multiplicites are
$$\tau_1 = \frac{n-1}{\alpha_2^2 - \alpha_1^2}\left(\alpha_2^2 - \frac{n-m}{m(n-1)} \right)
 \text{ and } 
\tau_2 = \frac{n-1}{\alpha_1^2 - \alpha_2^2}\left(\alpha_1^2 - \frac{n-m}{m(n-1)} \right).
$$
\end{prop}
}

For more information about frame theory and its applications, we refer to \cite{MR2964005}.

\subsection{Group theory and character theory}
Throughout this document, we let $\mathbb G$ denote a finite abelian group of order $n$, and we notate its group operations additively and write its identity element as $0_G$. Furthermore, we appeal to the fundamental theorem of abelian groups and  interpret $\mathbb G$ via some fixed cyclic decomposition
$$
\mathbb G = \mathbb Z_{n_1} \oplus ...  \oplus \mathbb Z_{n_k},
$$
where  $n_1,...,n_k$ are positive integers and where $\mathbb Z_{n_j}$ denotes the additive cyclic group of the integers modulo $n_j$.  Accordingly, whenever $k \geq 2$, we write each element $x \in \mathbb G$ as a $k$-tuple $x = (x(1),..,x(k))$; however, when $k=1$ so that $\mathbb G$ is a cyclic group, we drop the $k$-tuple notation and simply write its elements as $0=0_G, 1, 2,..., n-1$, in which case we write $a \equiv_n b$ to indicate that $a$ and $b$ both belong to the same congruence class of integers modulo $n$.   Furthermore, if $\mathbb G = \mathbb Z_p$ for some prime $p$, then
we write $\mathbb Z_p^*$ to denote the multiplicative group of $p-1$ nonzero elements in $\mathbb Z_p$, and we notate the multiplicative operations by juxtaposition or exponentiation.

  A {\bf character} for $\mathbb G$ is a homomorphism $\rho: \mathbb G \rightarrow \mathbb S^1$, where $\mathbb S^1$ denotes the complex unit circle endowed with standard multiplication.  The set 
$$\hat{ \mathbb G} = \{\rho : \rho \text{ is a character of } \mathbb G \}$$
is called the {\bf dual group of $\mathbb G$}, or occasionally just the {\bf dual group}.  The dual group for $\mathbb G$ also forms a group, where the group operation  is defined by pointwise multiplication and the identity element is the trivial map. 

A standard result in character theory is that every abelian group is isomorphic to its {\bf dual group} \cite{MR1440858}.  In general, there is no natural choice for an isomorphism taking $\mathbb G$ to $\hat{\mathbb G}$; however, with respect to our fixed cyclic decomposition for $\mathbb G$, we define the explicit isomorphism
$$
\Phi : \mathbb G \rightarrow \widehat{\mathbb G} 
\text{ by }
\Phi (x) = \rho_x,
$$
where 
$$
\rho_x (y) = e^{2 \pi i \left( \frac{x(1) y(1)}{n_1}  +\, \cdots \, +  \frac{x(k) y(k)}{n_k}   \right) }
$$
for each $y = (y(1),...,y(k)) \in \mathbb G$.  It is straightforward to check that $\Phi$ is a homomorphism with a trivial kernel, and since $|\mathbb G| = |\hat{\mathbb G}|$, it follows that it is an isomorphism, which means its image is $\widehat{\mathbb G}$. Throughout the rest of this document, we index the elements of $\widehat{\mathbb G}$ with the elements of $\mathbb G$ as determined by the definition of $\Phi$.  

The reason we have fixed the cyclic decomposition of $\mathbb G$ and indexed its characters as described above is so that the labeling of the elements of $\widehat{\mathbb G}$ satisfies the following conjugation and symmetry properties, which we record as a lemma.
\begin{lemma}\label{lem_lab_props}
For every $x, y \in \mathbb G$, we have
$$
\rho_x(y) = \rho_y(x) \text{ and } \overline{\rho_x(y)} = \rho_{-x} (y) = \rho_x( -y).
$$
\end{lemma}
\begin{proof}
This follows directly from the definition of $\Phi$.
\end{proof}

Given a subgroup $\mathbb H$ of $\mathbb G$,
the {\bf annihilator (subgroup) of $\mathbb H$ with respect to $\mathbb G$} is the set
$$\ANN = \{ \rho \in \widehat{\mathbb G} : \rho(x)=1 \text{ for each } x \in \mathbb H \}, $$
which forms a subgroup of $\hat{\mathbb G}$.
The following standard result is essential.
\begin{prop}\label{prop_sumchar_on_H}
If $\mathbb H$ is a subgroup of $\mathbb G$ and $\rho \in \widehat{\mathbb G}$, then
$$
\sum\limits_{x \in \mathbb H} \rho(x) =  \left\{ \begin{array}{cc}
| \mathbb H | , & \rho \in \ANN
\\
0 , & \text{otherwise}
\end{array}
\right. .
$$
\end{prop}
\begin{proof}
If $\rho \in \ANN$, then $\rho(x)=1$ for all $x \in \mathbb H$, so the claim follows for this case;  otherwise, there exists $y \in \mathbb H$ such that $\rho(y) \neq 1$, so we have 
$$
\sum\limits_{x \in \mathbb H} \rho(x) = \sum\limits_{x \in \mathbb H} \rho(y+x) 
=\rho(y) \sum\limits_{x \in \mathbb H} \rho(x), 
$$
which shows that 
$
\sum_{x \in \mathbb H} \rho(x) =0.
$
\end{proof}
If we let $\mathbb H = \mathbb G$ in this proposition, then with respect to our labeling of the characters, we have the following identity.

\begin{cor}\label{cor_gen_sum_roots_unity}
Let $x \in \mathbb G$, then
$$
\sum\limits_{y \in \mathbb G} \rho_x (y) =  \left\{ \begin{array}{cc}
n , & x = 0_G
\\
0 , & \text{otherwise}
\end{array}
\right. .
$$
\end{cor}

 Since every subgroup of an abelian group is normal, we may form the quotient group $\mathbb G / \mathbb H$,
so let $\pi: \mathbb G \rightarrow \mathbb G / \mathbb H$ be the corresponding quotient map.
It is straightforward to verify that the map
$$
\Psi: \widehat{\mathbb G / \mathbb H} \rightarrow\ANN
\text{
 defined by } \Psi(\rho) = \rho \circ \pi$$
 is a well-defined, surjective homomorphism with a trivial kernel, so $\ANN$ and $\widehat{\mathbb G / \mathbb H}$  are isomorphic.  In particular, since $\widehat{\mathbb G / \mathbb H}$ and ${\mathbb G / \mathbb H}$ are also isomorphic, we obtain the cardinality of $\ANN$.

\begin{prop}\label{prop_size_ann}
If $\mathbb H$ is a subgroup of $\mathbb G$ with order $l$, then
$$
|\ANN| = \left|  \widehat{\mathbb G \backslash \mathbb H}\right| =  |  {\mathbb G \backslash \mathbb H}| = 
\frac{|\mathbb G|}{|\mathbb H|}
=
\frac{n}{l}.
$$
\end{prop}
%$$
%\ker(\Psi_{\mathbb G, \mathbb H}) = \{ \rho \in \widehat{\mathbb G} : \rho(x)=1 \text{ for each } x \in \mathbb H \} = \frac{m}{l}.
%$$

For more information about the character theory of finite abelian groups, we refer to \cite{MR1440858}.

\subsection{$\mathbb G$-Harmonic Frames}
We conclude this section by defining and stating basic facts about harmonic frames, which have also been studied in \cite{MR2446568, MR2784566}, for example.

Given a subset $\mathcal S = \{g_1, ..., g_m\} \subset \mathbb G$, then the set of $n$ vectors
$$
\mathcal F~=~\{ f_x \}_{x \in \mathbb G} \subset \mathbb C^m,
$$
where
$$
f_x = \frac{1}{\sqrt{m}}\left( \rho_{g_j} (x) \right)_{j=1}^m \in \mathbb C^m,\text{  for each } x \in \mathbb G,
$$
is called the {\bf $\mathbb G$-harmonic frame generated by $\mathcal S$}, or sometimes just a {\bf harmonic frame}.  To see that $\mathcal F$ is indeed a frame, observe that the vectors have entries with constant magnitude $\frac 1 {\sqrt m}$, so they are unit norm, and Lemma~\ref{lem_lab_props} shows that the
%the
$(a,b)$-entry of 
%the frame operator 
the matrix $
\sum\limits_{x \in \mathbb G} f_x \otimes f_x^*
$ 
%of $\mathcal F$
 is
%\begin{align*}
%S_{a,b}
%&=
$$
\left(
\sum\limits_{x \in \mathbb G} f_x \otimes f_x^*
\right)_{a,b} 
=
\frac{1}{m} \sum\limits_{x \in \mathbb G}
 \rho_{g_a} (x) \overline{ \rho_{g_b} (x)}
=
\frac{1}{m} \sum\limits_{x \in \mathbb G}
 \rho_{g_a - g_b} (x) 
=\frac{n}{m} \delta_{a,b},
$$
where $\delta_{a,b}$ denotes the Kronecker delta function.
It follows that $\mathcal F$ spans $\mathbb C^m$ and is, in fact, a unit norm tight frame frame for $\mathbb C^m$ consisting of $n$ vectors.  
Another application of Lemma~\ref{lem_lab_props} shows that the distribution of the frame angles are completely detemined by the inner products that $f_{0_{G}}$ makes with the other frame vectors, because
$$
\langle f_x, f_y \rangle = \frac 1 m \sum\limits_{j=1}^m \rho_{g_j} (x) \overline{\rho_{g_j} (y)}= 
\frac 1 m \sum\limits_{j=1}^m \rho_{g_j} (x-y)  = \langle f_{x-y}, f_{0_G} \rangle
$$
for every $x,y \in \mathbb G$.  In particular, this shows that $\mathcal F$ is equidistributed.

Finally, we remark that, because their definition depends on the theory of characters, we typically interpret harmonic frames as objects which exist in complex Hilbert spaces; however, for certain groups, it is possible to choose $m$ distinct real-valued characters which can therefore be used to construct harmonic frames for $\mathbb R^m$ \cite{MR2784566}.  For example, if $\mathbb G = \mathbb Z_2^{k}$, the direct sum of $k$ copies of $\mathbb Z_2$, then it is straightforward to check that every character is real-valued, so we can interpret every harmonic frame derived from this group as a frame in a real Hilbert space.

\section{$\mathbb G$-Modulation Operators}\label{sec_mod}
In this section, we define and study the {\it $\mathbb G$-modulation operators} of frames, which will serve as essential tools for the analysis of biangular harmonic frames in the next section.  Modulation operators have recently been used to study certain types of optimal line packings, for example, maximal sets of mutually unbiased bases \cite{BandyopadhyayBoykinRoychowdhuryVatan2002} or the picket fence packings in \cite{MR3557826}.

Let  $\mathcal F = \{f_g \}_{g \in \mathbb G}$ be any frame for $\mathbb F^m$ whose vectors are indexed by the elements of $\mathbb G$.  For each $\xi \in \mathbb G$, we define  the {\bf $\xi$-th $\mathbb G$-modulation operator for $\mathcal F$} by
$$
X_\xi = \sum\limits_{x \in \mathbb G} \rho_\xi (x) f_x \otimes f_x^*,
$$
which can be thought of as the $\xi$-th value of the Fourier transform of the operator-valued map 
$$x \mapsto f_x \otimes f_x^*.$$
Direct substitution shows that the Fourier inversion formula holds and is given by
$$
 f_x \otimes f_x^* = \frac{1}{n} \sum\limits_{\xi \in \mathbb G} \rho_x (-\xi) X_\xi, \text{ for each } x \in \mathbb G.
$$ 
Given $x,y \in \mathbb G$, the Fourier inversion formula and the fact that
 $|\langle f_x, f_y \rangle |^2 = \TR(f_x \otimes f_x^* f_y \otimes f_y^*)$ shows that the  absolute values of the inner products between the frame vectors relate to the Hilbert Schmidt inner products between the 
modulation operators by
\begin{equation}\label{eq_encode_angles}
n^2 |\langle f_x, f_y \rangle |^2 = \sum\limits_{\xi, \eta \in \mathbb G} \rho_x(-\xi) \rho_y( \eta) \langle X_\xi, X_\eta \rangle_{H.S.}
.
\end{equation}

Whenever $\mathcal F$ is a $\mathbb G$-harmonic frame, then its  $\mathbb G$-modulation operators are Hilbert Schmidt orthogonal.  In order to see this, we first compute the entries of the modulation operators.
\begin{lemma}\label{prop_mod_entries}
If $\mathcal F = \{f_g\}_{g \in \mathbb G}$ is a $\mathbb G$-harmonic frame  for $\mathbb C^m$ generated by $\{g_1,...,g_m\}$, then the $(a,b)$-entry of its $\xi$-th $\mathbb G$-modulation operator
is
$$
\left(X_\xi \right)_{a,b} =
\left\{
\begin{array}{cc}
n/m, & g_b - g_a = \xi \\
0, & \text{otherwise}
\end{array}
\right. .
%
% \frac{N}{M} \delta_{\xi, g_b - g_a}.
$$
\end{lemma}

\begin{proof}
By definition, we have
\begin{align*}
(X_\xi)_{a,b} =  \sum\limits_{g \in \mathbb G} \rho_\xi (g) (f_g \otimes f_g^*)_{a,b}
=
\frac{1}{m}
\sum\limits_{g \in \mathbb G} \rho_\xi (g) \rho_{g_a} (g) \overline{\rho_{g_b}(g)}.
\end{align*}
By Lemma~\ref{lem_lab_props}, this simplifies to
$$
\frac{1}{m}
\sum\limits_{g \in \mathbb G} \rho_\xi (g) \rho_{g_a} (g) \overline{\rho_{g_b}(g)}
=
\frac{1}{m}
\sum\limits_{g \in \mathbb G} \rho_{\xi + g_a - g_b} (g) ,
$$
so the claim follows by Corollary~\ref{cor_gen_sum_roots_unity}.
\end{proof}

\begin{prop}\label{cor_hs_orth}
If $\mathcal F = \{f_g\}_{g \in \mathbb G}$ is a $\mathbb G$-harmonic frame  for $\mathbb C^m$ generated by $\{g_1,...,g_m\}$, then
$$
\langle X_\xi, X_\eta \rangle_{H.S.} =0 
$$
for every $ \xi, \eta \in \mathbb G$ with $\xi \neq \eta$. 
\end{prop}
\begin{proof}
By Lemma~\ref{prop_mod_entries}, for each $\xi \in \mathbb G$,  the nonzero entries of $X_\xi$ are indexed by the set $\{ (a,b) : g_b - g_a = \xi \}$.  Thus, if $(X_\xi)_{a,b} \neq 0$ for some $\xi \in \mathbb G$, then $(X_\eta)_{a,b} = 0$ for all $\eta \in \mathbb G$ with $\eta \neq \xi$, which implies that the modulation operators are pairwise Hilbert Schmidt orthogonal, as claimed.  
\end{proof}

In the special case that $\mathcal F$ is a $\mathbb G$-harmonic frame, the orthogonality between its $\mathbb G$-modulation operators  along with the symmetric and conjugate properties from Lemma~\ref{lem_lab_props} show that Equation~(\ref{eq_encode_angles}) simplifies to
\begin{equation}\label{eq_encode_angles_simp}
 n^2 |\langle f_x, f_y \rangle |^2 = \sum\limits_{\xi \in \mathbb G} \rho_{y-x} (\xi) \| X_\xi \|^2_{H.S.}
\end{equation}
for every $x, y \in \mathbb G$.

\section{Harmonic frames and nested difference sets}\label{sec_btfs}
In order to study and classify the  frame angle sets of harmonic frames, we will consider several known classes of combinatorial objects, including difference sets \cite{MR2246267}, divisible difference sets \cite{MR1440858}, relative difference sets \cite{MR1440858}, partial difference sets \cite{MR1277942}, and almost difference sets \cite{2014arXiv1409.0114N}.  These families of objects are closely related, as they all fall under our umbrella term, {\it bidifference sets}.  In fact, as can be verified by examples from \cite{MR1277942,2014arXiv1409.0114N}, these classes of structures are not mutually exclusive.  However, each type of bidifference set just listed is described with its own system of notation; moreover, our work concludes with an examination of a combinatorial structure which is not a bidifference set.  Thus, in order to describe all of these objects with a unified notation, we begin by defining a more general structure, a {\it nested difference set.}

\begin{defn}
An $m$-element subset $\mathcal S= \{g_1,...,g_m \} \subset \mathbb G$ is called an
{\bf $(n,m,t)$-nested difference set} for $\mathbb G$ relative to $(\mathcal A, \Lambda)$ if there exists a sequence of subsets  $\mathcal{A} = \{A_j\}_{j =0}^t$ of $\mathbb G$ and a sequence of nonnegative integers $\Lambda = \{ \lambda_j\}_{j=1}^t$ such that
\begin{enumerate}
\item
$A_0=\{0_G\}$,
\item
$A_t = \mathbb G$,
\item
the sequence $\mathcal A$ is increasing with resepect to set containment,
$$
A_0 = \{0_G\} \subset A_1 \subset ... \subset A_{t-1} \subset A_t = \mathbb G,
$$
\item
and, for every $j\in\{1,2,...,t\}$,  each element $x \in A_j \backslash A_{j-1}$ can be expressed as $x=g_a - g_b$ in exactly $\lambda_j$ ways.
\end{enumerate}
We call the multiset $\left\{g_a - g_b : 1 \leq a,b \leq m, a \neq b \right\}$
the {\bf difference structure of $\mathcal S$}.
\end{defn}

\begin{ex}\label{ex_diff_73}
Let $\mathbb G=\mathbb Z_7$.  The set $\mathcal S= \{0,1 , 3\} \subset \mathbb Z_7$ is a $(7, 3, 1)$-nested difference set for $\mathbb G$ 
relative to $\left( \{A_0,A_1 \}, \{\lambda_1\} \right)$, where $A_0 = \{0\}$ and $A_1 = \mathbb G$ and where $\lambda_1=1$, because $0-1=6, 1-0=1, 0-3=4, 3-0=3, 3-1=2$ and $1-3=5$.
\end{ex}

\begin{ex}
Let $\mathbb G=\mathbb Z_7$.  The set $\mathcal S= \{0,1,2\} \subset \mathbb Z_7$ is a $(7, 3, 3)$-nested difference set for $\mathbb G$ 
relative to $\left( \{A_0,A_1,A_2,A_3 \}, \{\lambda_1, \lambda_2, \lambda_3\} \right)$, where $A_0 = \{0\}, A_1=\{ 0, 2 , 5\}, A_2= \{0, 1, 2, 5, 6  \}$ and $A_3 = \mathbb G$ and where $\lambda_1=1, \lambda_2=2$ and $\lambda_3 =0$, because $0-1=6, 1-0=1, 0-2=5, 2-0=2, 1-2=6$ and $2-1=1$.
\end{ex}

\begin{ex}\label{ex_nest_div_833}
Let $\mathbb G=\mathbb Z_2 \oplus \mathbb Z_4$.  The set $\mathcal S= \left\{(0,0), (1,0), (0,1) \right\} \subset \mathbb G$ is an $(8, 3, 3)$-nested  difference set for $\mathbb G$ relative to $\left(\{ A_j \}_{j=0}^3, \{\lambda_j \}_{j=1}^3 \right)$, where
$A_0 = \{0_G \}$,
$A_1 = \{ (0,0),(1,0)\}$,
$A_2 = \{ (0,0),(1,0), (0,2), (1,2)\}$
and $A_3 =\mathbb G$,
and where
$\lambda_1 =2, \lambda_2 =0$, and $\lambda_3=1$, because 
$(0,0) -(1,0) = (1,0), 
(1,0) -(0,0) = (1,0), 
(0,0) -(0,1) = (0,3), 
(0,1) -(0,0) = (0,1), 
(1,0) -(0,1) = (1,3)$
and
$(0,1) -(1,0) = (1,1).$
\end{ex}

If $\mathcal S$ is a $(n,m,t)$-nested difference set for $\mathbb G$ relative to $(\mathcal A, \Lambda)$, then, by appending copies of the last elements of $\mathcal A$ and $\Lambda$, for example, it is clear that $\mathcal S$ is an $(n,m,t')$-nested difference set for every $t' \geq t$.

\begin{defn}
Let $\mathcal S$ be an $(n,m,t)$-nested difference set for $\mathbb G$ relative to $(\mathcal A, \Lambda)$ for some $t\geq 2$. 
 We say that $\mathcal S$ is a {\bf proper $(n,m,t)$-nested difference set for $\mathbb G$} if there does not exist a pair $(\mathcal A', \Lambda')$ such that  $\mathcal S$ is an $(n,m,t-1)$-nested difference set for $\mathbb G$ relative to $(\mathcal A', \Lambda')$.
\end{defn}
It is not difficult to see that every $m$-element subset $\mathcal S \subset \mathbb G$ is an $(n,m,t)$-nested difference set with respect to some pair $(\mathcal A, \Lambda)$ and some integer $t \geq 1$, so this definition is only interesting when equipped with additional structure.  For the remainder of this paper, we study various types of nested difference sets  and the relationship they have with the harmonic frames that they generate.  We begin with the {\it difference sets}.

%\begin{rem} 
%  In particular, every difference set is also a bidifference set.
%\end{rem}
\subsection{Difference Sets}
\begin{defn}
Let $\mathcal S \subset \mathbb G$.  We say that $\mathcal S$ is 
an {\bf $(n,m, \lambda)$-difference set for $\mathbb G$} if it is an $(n,m,1)$-nested difference set for $\mathbb G$ relative to $(\{A_0,A_1\}, \{\lambda_1\})$, where  $\lambda=\lambda_1$.
\end{defn}

Difference sets have been studied extensively \cite{MR0001221, MR2246267, MR1440858}, and it is well-known that equiangular harmonic frames are characterized by them \cite{MR1984549, XiaZhouGiannakis2005, MR2446568}.

\begin{thm}\label{thm_char_ETFs}[\cite{MR1984549, XiaZhouGiannakis2005, MR2446568}]
If $\mathcal F = \{f_g\}_{g \in \mathbb G}$ is a $\mathbb G$-harmonic frame  for $\mathbb C^m$ generated by  $\mathcal S = \{g_1,...,g_m\}$, then $\mathcal F$ is equiangular if and only if there exists some positie integer $\lambda$ such that $\mathcal S$ is an $(n, m, \lambda)$ difference set for $\mathbb G$.
\end{thm}

\begin{ex}
The nested difference set from Example~\ref{ex_diff_73} is a $(7,3,1)$-difference set for $\mathbb G$, so by the preceding theorem, it generates an equiangular tight $\mathbb G$-harmonic frame for $\mathbb C^3$ consisting of $7$ vectors. See~(1.6) of~\cite{2016arXiv160909836I} for an explicit construction of this frame.
\end{ex}

Because the relationship between difference sets and harmonic frames is already well-understood, we continue our investigation into structured  $(n,m,t)$-nested difference sets by considering the cases where $t=2$.  We call these {\it bidifference sets}.
\subsection{Bidifference Sets}

\begin{defn}
Let $\mathcal S \subset \mathbb G$.  We say that $\mathcal S$ is  an {\bf $(n,m,l, \lambda, \mu)$-bidifference set for $\mathbb G$ relative to $A$} if it is an $(n,m,2)$-nested difference set for $\mathbb G$ relative to $(\{A_0,A_1, A_2\}, \{\lambda_1, \lambda_2\})$, where  $\lambda=\lambda_1$, $\mu = \lambda_2$, $l=|A_1|$ and $A=A_1$.  We say that $\mathcal S$ is a {\bf proper $(n,m,l, \lambda, \mu)$-bidifference set for $\mathbb G$} if it is not an $(n,m,\lambda)$-difference set for $\mathbb G$.
\end{defn}

%Next, we examine the properties of harmonic frames generated by bidifference sets.   We will see that there are numerous infinite  families of bidifference sets that generate BTFs, so naively one might  expect a  correspondence between bidifference sets and BTFs similar to the relationship between difference sets and ETFs from Theorem~\ref{thm_char_ETFs}.  However, we will see there exist bidifference sets that do not generate equiangular harmonic frames

 We begin our examination
% of the relationship between harmonic frames and bidifference sets 
%
with a computation that further simplifies Equation~(\ref{eq_encode_angles_simp}) in the special case that a harmonic frame is generated by a bidifference set.

\begin{prop}\label{prop_angles_simp2}
 If $\mathcal F = \{f_g\}_{g \in \mathbb G}$ is a $\mathbb G$-harmonic frame  for $\mathbb C^m$ generated by an $(n,m,l, \lambda, \mu)$-bidifference set $\mathcal S = \{g_1,...,g_m\}$ for  $\mathbb G$ relative to $A$, then
\begin{equation}\label{eq_angle_simp2}
m^2 |\langle f_x, f_y \rangle |^2
=
(m-\lambda)+ (\lambda - \mu) \sum\limits_{\xi \in  A }\rho_{y-x} (\xi)
\end{equation}
for every $x,y \in \mathbb G$ with $x \neq y$.  
\end{prop}
\begin{proof}
By Lemma~\ref{prop_mod_entries}, the squared Hilbert Schmidt norm of the $\xi$-th modulation operator is
$$
\|X_\xi\|_{H.S.}^2
=
\left\{
\begin{array}{cc}
\frac{n^2}{m}, & \xi=0 \\
\lambda \frac{n^2}{m^2}, & \xi \in  A \backslash \{0\}\\
\mu \frac{n^2}{m^2}, & \xi \in \mathbb G \backslash A\\
\end{array}
\right.,
$$
so if $x, y \in \mathbb G$ with $z=y-x \neq 0$, then Equation~(\ref{eq_encode_angles_simp}) can be rewritten as
\begin{align*} m^2 |\langle f_x, f_y \rangle |^2 &=
m
+
\lambda 
\sum\limits_{\xi \in A \backslash \{0\}}
\rho_{z} (\xi)
+
\mu
\sum\limits_{\xi \in \mathbb G \backslash  A}
\rho_{z} (\xi)\\
&=
(m-\lambda)+ (\lambda - \mu) \sum\limits_{\xi \in A }\rho_{z} (\xi)
+
  \mu  \sum\limits_{\xi \in \mathbb G}\rho_{z} (\xi).
\end{align*}
The third term vanishes
by Corollary~\ref{cor_gen_sum_roots_unity},
so Equation~(\ref{eq_angle_simp2}) follows.
\end{proof}

As a direct consequence of this computation, we obtain a characterization for the cardinality of the frame angle set of any harmonic frame generated by a bidifference set.

\begin{cor}\label{cor_char_kangular}
If $\mathcal F = \{f_g\}_{g \in \mathbb G}$ is a $\mathbb G$-harmonic frame  for $\mathbb C^m$ generated by an $(n,m,l,\lambda,\mu)$-bidifference set $\mathcal S = \{g_1,...,g_m\}$ for  $\mathbb G$ relative to $A$, then $\mathcal F$ is $k$-angular if and only if
$$
\left|
\left\{
 \sum\limits_{\xi \in  A }\rho_{z} (\xi) : z \in \mathbb G \backslash \{0_G\}
\right\}
\right|=k.
$$
%\begin{enumerate}
%\item
% $\mathcal F$ is an equiangular tight frame if and only if $\mathcal S$ is an $(n,m, \lambda)$-difference set for $\mathbb G$.
%\item
%$\mathcal F$ is a proper biangular tight frame if and only if
%$$
%\left|
%\left\{
% \sum\limits_{\xi \in \mathcal A }\rho_{z} (\xi) : z \in \mathbb G \backslash \{0_G\}
%\right\}
%\right|=2.
%$$
%\end{enumerate}
\end{cor}
%\begin{proof}
%The claim about equiangular tight frames is just a restatement of Theorem~\ref{thm_char_ETFs}, and the claim about biangular tight frames follows directly from Proposition~\ref{prop_angles_simp2}.  
%
%\end{proof}

\begin{ex}\label{ex_bds_div}
Let $\mathbb G= \mathbb Z_6$.   The set $\mathcal S = \{0,1,3 \} \subset \mathbb G$ is a $(6,3,2,2,1)$-bidifference set for $\mathbb G$ relative to $A$, where $A= \{ 0,3\}$.  If $z\in \mathbb G$ with $z \neq 0$, then
$$
\sum\limits_{\xi \in A }\rho_{z} (\xi)
=
\sum\limits_{\xi=0,3} e^{2 \pi i \xi z /6 }
=
\left\{
\begin{array}{cc}
0, & z=1,3,5\\
2, & z=2,4\\
\end{array}
\right.,
$$
so Corollary~\ref{cor_char_kangular} shows that the $\mathbb G$-harmonic frame generated by $\mathcal S$ is a biangular tight frame.
\end{ex}

\begin{ex}\label{ex_bds_counter}
Let $\mathbb G= \mathbb Z_9$.   The set $\mathcal S = \{0,1,3,4 \} \subset \mathbb G$ is an $(9,4,4,2,1)$-bidifference set for $\mathbb G$ relative to $A$, where $A= \{ 0,1,3,6,8\}$.  If $z\in \mathbb G$ with $z \neq 0$, then 
$$
\sum\limits_{\xi \in  A }\rho_{z} (\xi)
=
\sum\limits_{\tiny \xi=0,1,3,6,8} e^{2 \pi i \xi z /9}
=
\left\{
\begin{array}{cc}
2 \cos(2 \pi i/9), & z=1,8\\
2 \cos(4 \pi i/9), & z=2,7\\
2, & z=3,6\\
2 \cos(8 \pi i/9), & z=4,5\\
\end{array}
\right.,
$$
so  the $\mathbb G$-harmonic frame generated by $\mathcal S$ is $4$-angular by  Corollary~\ref{cor_char_kangular}.
\end{ex}

As demonstrated in Example~\ref{ex_bds_counter}, there exist bidifference sets which do not generate biangular harmonic frames, so we cannot expect a  relationship between bidifference sets and harmonic BTFs as nice as the correspondence between difference sets and harmonic ETFs from Theorem~\ref{thm_char_ETFs}.  Nevertheless, we will see that there are numerous classes of structured bidifference sets which  always generate either biangular or equiangular harmonic frames.  Example~\ref{ex_bds_div} is an example from one of these families, the {\it divisible difference sets}.

\subsubsection{Divisible difference sets}
\begin{defn}\label{def_dds}
Let $\mathcal S \subset \mathbb G$ and let $\mathbb H$ be a subgroup of $\mathbb G$ of order $l$.  We say that $\mathcal S$ is  an {\bf $(n,m,l, \lambda, \mu)$-divisible difference set for $\mathbb G$ relative to $\mathbb H$} if it is an $(n,m,l, \lambda, \mu)$-bidifference set for $\mathbb G$ relative to $\mathbb H$.  
\end{defn}

\begin{rem}\label{rem_DDS_RDS_notn}
For typographical reasons and because our main interest is in the order of $\mathbb G$, we have deviated from  the standard notation among the experts \cite{MR1440858, MR1232066} for the definitions of divisible difference sets and the special case of relative difference sets (Definition~\ref{def_rds}).  Typically, the first, second, and third parameters are, respectively, the index of underlying subgroup $\mathbb H$ in $\mathbb G$, the order of $\mathbb H$ and the order of the subselection $\mathcal S$, in which case the size $n$ of the group becomes redundant according to Lagrange's theorem.  
\end{rem}

\begin{thm}\label{thm_dds_btf}
If $\mathcal F = \{f_g\}_{g \in \mathbb G}$ is a $\mathbb G$-harmonic frame  for $\mathbb C^m$ generated by an $(n,m,l, \lambda, \mu)$-divisible difference set $\mathcal S = \{g_1,...,g_m\}$ for  $\mathbb G$ relative to $\mathbb H$, then either
\begin{enumerate}
\item
$\lambda = \mu$ and $\mathcal F$ is an equiangular tight frame, or
\item
$\lambda \neq \mu$ and
$\mathcal F$ is a biangular tight frame with frame angles
$$
\alpha_1 
=
\frac{1}{m}
\sqrt
{
{m- \lambda + l  (\lambda - \mu)}
}
\text{ and }
\alpha_2 = 
\frac{1}{m}\sqrt{{m- \lambda}}
$$
and frame angle multiplicities
$$
\tau_1
=
\frac{n}{l}
\text{ and }
\tau_2 = n - \frac{n}{l} -1.
$$
\end{enumerate}
\end{thm}

\begin{proof}
If $\lambda=\mu$, then $\mathcal S$ is an $(n,m,\lambda)$-difference set for $\mathbb G$, so it follows that $\mathcal F$ is equiangular by Theorem~\ref{thm_char_ETFs}.  If $\lambda \neq \mu$, then,
by  Proposition~\ref{prop_sumchar_on_H}, we have
$$
 \sum\limits_{\xi \in \mathbb H }\rho_{z} (\xi)
=\left\{ \begin{array}{cc}
l , & \rho_z \in \ANN
\\
0 , & \text{otherwise}
\end{array}
\right.,
$$
so $\mathcal F$ is a biangular tight frame by Corollary~\ref{cor_char_kangular}, 
and the claimed frame angle values are determined by substituting the value of this summation into Equation~(\ref{eq_angle_simp2}) from Proposition~\ref{prop_angles_simp2}.  The claimed values for the frame angle multipicities follow from Proposition~\ref{prop_size_ann}.
\end{proof}

Numerous infinite families of divisible difference sets are known \cite{MR1440858}, thereby producing infinite families of BTFs according to the previous theorem.  In Table~\ref{tbl_dds} of the appendix, we list several known families of divisible difference sets along with information about the frame angle sets of the corresponding harmonic frames. 

 A type of divisible difference set which has received special attention \cite{MR2057609, GodsilRoy2009, 2014arXiv1409.0114N} occurs when the difference structure includes no nonzero elements from the relative subgroup.  Because this class has been particularly useful for the construction of certain optimal line packings \cite{GodsilRoy2009, 2014arXiv1409.0114N}, we state its definition here and separately tabulate examples for this special case in Table~\ref{tbl_rds} of the appendix.

\begin{defn}\label{def_rds}
Let $\mathcal S \subset \mathbb G$ and let $\mathbb H$ be a subgroup of $\mathbb G$ of order $l$.  We say that $\mathcal S$ is  an {\bf $(n,m,l,  \mu)$-relative difference set for $\mathbb G$ relative to $\mathbb H$} if it is an $(n,m,l,0,\mu)$-divisible difference set for $\mathbb G$ relative to $\mathbb H$.  
\end{defn}

We restate Theorem~\ref{thm_dds_btf} for the special case that a harmonic frame is generated by a relative difference set.
\begin{cor}\label{cor_rds_btf}
If $\mathcal F = \{f_g\}_{g \in \mathbb G}$ is a $\mathbb G$-harmonic frame  for $\mathbb C^m$ generated by an $(n,m,l, \mu)$-relative difference set $\mathcal S = \{g_1,...,g_m\}$ for  $\mathbb G$ relative to $\mathbb H$, then either
\begin{enumerate}
\item
$\mathbb H=\{0_G\}$, so that $\mathcal S$ is an $(n,m,\mu)$-difference set for $\mathbb G$ and $\mathcal F$ is an equiangular tight frame (by Theorem~\ref{thm_char_ETFs}), or
\item
$\mathbb H\neq \{0_G\}$  and
$\mathcal F$ is a biangular tight frame with frame angles
$$
\alpha_1 
=
\frac
{\sqrt
{
{m-l\mu}
}} {m}
\text{ and }
\alpha_2 = 
\frac{\sqrt{m}}{m}
$$
and frame angle multiplicities
$$
\tau_1
=
\frac{n}{l}
\text{ and }
\tau_2 = n - \frac{n}{l} -1.
$$
\end{enumerate}
\end{cor}

Next, we consider a type of bidifference set which does not require  an underlying subgroup over which its difference structure partitions; instead,  the defining property is that  it forms a bidifference set relative to itself adjoined with the zero element.  Such an object is called a {\it partial difference set}.  \jh{Partial difference sets have connections with strongly regular graphs \cite{MR1277942}, two-weight codes \cite{MR1277942}, and other interesting mathematical objects \cite{MR1277942, MR2064754, MR765286}.  They also generate harmonic frames which are either equiangular or biangular.}

\subsubsection{Partial difference sets}\label{sec_pds}

\begin{defn}\label{def_pds}
Let $\mathcal S \subset \mathbb G$.  We say that $\mathcal S$ is  an {\bf $(n,m,\lambda, \mu)$-partial difference set for $\mathbb G$} if 
$\mathcal S$ is an $(n,m,l, \lambda, \mu)$-bidifference set for $\mathbb G$ relative to $\mathcal S \cup \{ 0_G \}$,
where $l = |\mathcal S \cup \{ 0_G \}|$. 
% Moreover, we say an $(n,m,\lambda, \mu)$-partial difference set for $\mathbb G$ is {\bf regular} if $0_G \notin \mathbb G$.
\end{defn}
\jh{
\begin{rem}
The notation used for parametrization of partial difference sets has evolved over the last few decades; see \cite{MR765286} in comparison with \cite{MR1277942}.  The notation we have used in the preceding definition is consistent with what now seems to be the standard format \cite{MR1277942, MR2064754, MR2496253}.
\end{rem} }
In order to see that partial difference sets generate biangular harmonic frames, we collect a few results about partial difference sets from the literature \cite{MR1277942, MR765286}.  Given a subset $\mathcal S \subset \mathbb G$, we denote $-\mathcal S:=\{-g : g\in \mathcal S\}$, which is called the {\it reversal of $\mathcal S$.}
%Direct computation of the difference structure shows that partial difference sets come in pairs, as we state in the following corollary.
%	\begin{cor}\label{cor_pds_reg} [\cite{MR1277942,MR765286}.]
%		If $\mathcal S \subset \mathbb G$, then $\mathcal S$ is an $(n, m,\lambda, \mu)$ partial difference set for $\mathbb G$ with  $-\mathcal S=\mathcal S$ and $0_G \notin \mathcal S$ if and only if
%$\mathcal S \cup \{0_G\}$ is an $(n, m+1,\lambda+2, \mu)$ partial difference set for $\mathbb G$ with  $-\mathcal S=\mathcal S$.
%	\end{cor}
	\begin{prop}\label{prop_pds_inv} [Proposition 1.2 of \cite{MR1277942}; see also \cite{MR765286}]
		If $\mathcal S$ is an $(n, m,\lambda, \mu)$ partial difference set for $\mathbb G$ and $\lambda\neq\mu$ (ie, $\mathcal S$ is proper), then $-\mathcal S=\mathcal S$.
	\end{prop}
\jh{
\begin{defn}
Let $\mathcal S$ be  an $(n,m,\lambda, \mu)$-partial difference set for $\mathbb G$.  If $-{\mathcal S} = \mathcal S$, then we say that $\mathcal S$ is {\bf reversible}.  If $\mathcal S$ is a reversible and $0_G \notin \mathbb G$, then we say that $\mathcal S$ is {\bf regular}.
\end{defn}
}
We  recall the following computation regarding the summation of characters over proper partial difference sets.
	\begin{thm}\label{thm_pds_sumval}[\cite{MR1277942, MR2912882}]
		If $\mathcal S$ is an $(n,m, \lambda, \mu)$ partial difference set for $\mathbb G$
%a subset of $\mathbb G$ 
such that $-\mathcal S=\mathcal S$, then 
%$\mathcal S$ is an $(n,m, \lambda, \mu)$ partial difference set for $\mathbb G$ if and only if
		\begin{equation}\label{eq_pds_sum}
		\sum_{\xi\in \mathcal S}\rho_z(\xi)=
\left\{
\begin{array}{cc}
m, & z= 0_G \\
\dfrac{1}{2}\left({\lambda-\mu\pm\sqrt{(\lambda-\mu)^2+4\gamma}}\right),
&
\text{otherwise}
\end{array}
\right.,
\end{equation}
where $\gamma=m-\lambda$ if $0_G\in \mathcal S$ and $\gamma=m-\mu$ if $0_G\notin \mathcal S$.
	\end{thm}

Proposition~\ref{prop_pds_inv} together with Theorem~\ref{thm_pds_sumval} show that every  partial difference set generates either an ETF or a BTF. % proper harmonic BTF.

\begin{theorem}\label{cor_pds_btf}
If $\mathcal F = \{f_g\}_{g \in  G}$ is a $\mathbb G$-harmonic frame for $\mathbb C^m$  generated by $\mathcal S$, where $\mathcal S$ is an $( n,  m, \lambda, \mu)$-partial difference set for $\mathbb G$, then either
\begin{enumerate}
\item $\lambda=\mu$ and $\mathcal F$ is an equiangular tight frame, or
\item
$\lambda \neq \mu$, $0_G \in \mathcal S$ %(ie, $\mathcal S$ is proper but not regular) 
and
		$\mathcal F$ is a biangular tight frame with frame angles
$$
\alpha_1, \alpha_2=
\frac{\sqrt{2}}{2m}
\sqrt
{
2(m-\lambda)
+
(\lambda -\mu)
\left(
\lambda - \mu \pm
\sqrt{(\lambda - \mu)^2 
 + 4 (m-\lambda)}
\right)
},
$$
\item
or 
$\lambda \neq \mu$, $0_G \notin \mathcal S$ %(ie, $\mathcal S$ is proper and regular) 
and
		$\mathcal F$ is a biangular tight frame with frame angles
$$
\alpha_1, \alpha_2=
\frac{\sqrt{2}}{2m}
\sqrt
{
2(m-\mu) +(\lambda - \mu)\left(\lambda - \mu \pm \sqrt{(\lambda-\mu)^2 +4(m-\mu)}\right)
}.
$$
\end{enumerate}
\end{theorem}
\begin{proof}
If $\lambda=\mu$, then $\mathcal S$ is an $(n,m,\lambda)$-difference set for $\mathbb G$, so $\mathcal F$ is equiangular by Theorem~\ref{thm_char_ETFs}.  If $\lambda \neq \mu$, then $\mathcal S= -\mathcal S$ by Proposition~\ref{prop_pds_inv}, so Equation~(\ref{eq_pds_sum}) from Theorem~\ref{thm_pds_sumval} holds.  
If $0_G \in \mathcal S$, then $\mathcal F$ is a biangular tight frame by Corollary~\ref{cor_char_kangular}.
Substituting the value of the summation in  Equation~(\ref{eq_pds_sum}) directly into Equation~(\ref{eq_angle_simp2}) from Proposition~\ref{prop_angles_simp2} shows the claimed values of the frame angles.
If $0_G \notin \mathcal S$, then adjusting Equation~(\ref{eq_pds_sum}) from Theorem~\ref{thm_pds_sumval} yields
$$
\sum_{\scriptscriptstyle \xi\in \mathcal S \cup \{0_G\}}\rho_z(\xi)=
\left\{
\begin{array}{cc}
1+m, & z= 0_G \\
1+
\dfrac{1}{2}\left({\lambda-\mu\pm\sqrt{(\lambda-\mu)^2+4(m-\mu)}}\right),
&
\text{otherwise}
\end{array}
\right.,
$$
so again
$\mathcal F$ is biangular by Corollary~\ref{cor_char_kangular} and
substituting this sum
into Equation~(\ref{eq_angle_simp2}) from Proposition~\ref{prop_angles_simp2} shows the claimed values of the frame angles for this case.
\end{proof}
\jh{
\begin{rem}
In Theorem~\ref{thm_dds_btf}, we stated the frame angle multiplicities of the resulting biangular frame because they follow directly from the structure of the divisible difference set from which the frame is generated; however, we do not state these values in the preceding theorem, because, except for certain cases, these values are not as easily extrapolated from the structure of the underlying  proper partial difference set.  Nevertheless, this information becomes apparent after applying Propososition~\ref{prop_btf_multiplicities}.  In this sense, we find that the representation of a partial difference set as a harmonic frame sheds some light on the partial difference set's structure.
\end{rem}
}

\jh
{
The following fact implies that each harmonic BTF produced by a partial difference set is accompanied by a harmonic BTF(or ETF) in a complex vector space of dimension different from its own by one.
\begin{prop}[\cite{MR1277942}]
If $\mathcal S$ is a reversible $(n,m,\lambda, \mu)$ partial difference set for $\mathbb G$ such that $0_G \in \mathbb G$, then $\mathcal S \backslash \{0_G\}$ is a regular $(n, m-1, \lambda-2, \mu)$-partial difference set for $\mathbb G$.  Conversely, if $\mathcal S$ is a regular $(n,m,\lambda, \mu)$ partial difference set for $\mathbb G$, then $\mathcal S \cup \{0_G\}$ is a reversible $(n, m+1, \lambda+2, \mu)$-partial difference set for $\mathbb G$.
\end{prop}

\begin{cor}\label{cor_pds_btf_pairs}
Let $\mathcal S$ be a $(n,m,\lambda,\mu)$-partial difference set for $\mathbb G$ with $\lambda \neq \mu$ (so $\mathcal S$ is reversible) and let $\mathcal F$ be the biangular $\mathbb G$-harmonic frame for $\mathbb C^m$ generated by $\mathcal S$.  
\begin{enumerate}
\item
If $0_G \in \mathcal S$, then the $\mathbb G$-harmonic frame $\mathcal F'$ for $\mathbb C^{m-1}$ generated by $\mathcal S\backslash \{0_G\}$ is either equiangular or biangular.
\item
If $\mathcal S$ is regular, then the $\mathbb G$-harmonic frame $\mathcal F'$ for $\mathbb C^{m+1}$ generated by $\mathcal S\cup \{0_G\}$ is either equiangular or biangular.
%, as in Theorem~\ref{cor_pds_btf}.
\end{enumerate}
\end{cor}
}

Because of developments in the theory of partial difference sets in recent decades \cite{MR1277942, MR765286}, there are numerous known examples of partial difference sets which therefore produce biangular tight frames by Theorem~\ref{cor_pds_btf}.     In Table~\ref{tbl_pds} in the appendix, we list several infinite families of partial difference sets and the information about the angles sets of the corresponding harmonic frames. The following example, which is due to Paley \cite{MR1277942}, is obtained by generalizing the construction of certain cyclotomic difference sets. 

%\cite{} for more information about cyclotomic difference sets or see \cite{} for more details about the following example.

\begin{ex}\label{ex_quad}
Let $\mathbb G = \mathbb Z_p$, where $p$ is an odd prime, 
 and let 
$\RES{2} = \{ z^2 : z \in \mathbb Z_p^* \}$, the set of {\it quadratic residues} in $\mathbb Z_p^*$.  If $p \equiv_4 3$, then $\RES{2}$ is a $\left(p,\frac{p-1}{2},\frac{p-3}{4}\right)$-difference set for $\mathbb Z_p$ and therefore generates an equiangular harmonic frame by Theorem~\ref{thm_char_ETFs}; otherwise,  $p \equiv_4 1$ and $\RES{2}$ is a $\left(p, \frac{p-1}{2}, \frac{p-5}{4}, \frac{p-1}{4} \right)$-partial difference set for $\mathbb Z_p$ and therefore generates a biangular harmonic frame by the preceding theorem.
\end{ex}

In the next section, we examine a class of bidifference sets which can be viewed as a generalization of the partial difference sets from Example~\ref{ex_quad}.  We call them {\it Gaussian difference sets} because of their relationship with Gauss sums.

\subsubsection{Gaussian difference sets}

Given an odd prime $p$, we denote the  multiplicative subgroup of {\bf $s$-th residues in $\mathbb Z_p^*$}  by
$$
\RES{s} := \left\{z^s : z \in \mathbb Z_{p}^* \right\},
$$
and we recall from \cite{MR1625181} that $\left| \RES{2} \right| = \frac{p-1}{2}$.
\begin{defn}
Let $\mathbb G = \mathbb Z_p$, where $p$ is prime, and let $\mathcal S \subset \mathbb G$.  We say that $\mathcal S$ is  a {\bf $(p,m, \lambda, \mu)$-Gaussian difference set for $\mathbb G$}
if 
$\mathcal S$ is a $(p,m, \frac{p+1}{2}, \lambda, \mu)$-bidifference set for $\mathbb G$  relative to $\RES{2} \cup \{0\}$.
\end{defn}

Like divisible differnce sets and partial difference sets, every harmonic frame generated by a Gaussian difference set is either equiangular or biangular.  This fact depends on the theory of quadratic residues and Gauss sums.
%Before we prove Theorem~\ref{th_quartic_btfs2}, we  recall another result of Gauss.
Given any odd prime $p$ and $a \in \mathbb Z_p^*$, we define the {\bf quadratic residue (or Legendre) symbol}  by
$$(a/p)_2 \equiv_p a^{\frac{p-1}{2}}.$$

Euler provided the following characterization of the quadratic residues.

\begin{thm}\label{th_eul_crit}[Euler's criterion, \cite{MR1625181}]
If $p$ is an odd prime and $a \in \mathbb Z_p^*$, then
$$(a/p)_2 \equiv_p \left\{ \begin{array}{cc} 1 , & a \in \RES{2} \\ -1 , & a \notin \RES{2} \end{array} \right..
$$
\end{thm}

%More generally, given $a \in \mathbb Z_p$, where $p=q^k$ is an odd prime power, then we define the {\bf Jacobi symbol}  by
%$$
%(a/p)_J \equiv_p (a/q)_2^{k}. 
%$$
%
%\begin{thm}\label{th_leg}[Jacobi, \cite{}]
%If  $p=q^k$ is an odd prime power and $a \in \mathbb Z_p^*$, then
%$$(a/p)_J
%\equiv_p 
%\left\{ \begin{array}{cc} 1 , & a \in \RES{2} \\ -1 , & a \notin \RES{2} \end{array} \right..
%$$
%\end{thm}

In accordance with Euler's criterion, we assign the real value $1$ or $-1$ to the Legendre symbol if it appears in a numerical computation.  In particular, we use this idea in the next theorem, where we state how to evaluate quadratic Gaussian sums over cyclic groups of odd prime order.

\begin{thm}\label{th_gsum}[Quadratic Gauss Sums, \cite{MR1625181}]
If $p$ is an odd prime, then for each $a\in \mathbb Z_p^*$, we have
$$
\sum\limits_{x \in \mathbb Z_p} e^{\frac{2 \pi i a x^2}{p}} = 
\left\{
\begin{array}{cc}
(a/p)_2 \sqrt{p}, & p \equiv_4 1\\
(a/p)_2  \sqrt{p} i, & p \equiv_4 3\\
\end{array}
\right. .
$$
\end{thm}

\begin{cor}\label{cor_gauss_val}
If $p$ is an odd prime, then for each $a\in \mathbb Z_p^*$, we have
$$
\sum\limits_{j \in \RES{2} \cup \{0 \} } e^{\frac{2 \pi i a j}{p}} = 
\left\{
\begin{array}{cc}
\frac{1+\sqrt{p}}{2} & p \equiv_4 1, a \in \RES{2}\\
\frac{1-\sqrt{p}}{2} & p \equiv_4 1, a \notin \RES{2}\\
\frac{1+\sqrt{p}i}{2} & p \equiv_4 3, a \in \RES{2}\\
\frac{1-\sqrt{p}i}{2} & p \equiv_4 3, a \notin \RES{2}\\
\end{array}
\right. .
$$
\end{cor}
\begin{proof}
This follows by evaluating the Legendre symbol in the preceding theorem and the observation that
$$
\sum\limits_{x \in \mathbb Z_p} e^{\frac{2 \pi i a x^2}{p}}
=
2 \sum\limits_{\tiny j \in \RES{2} \cup \{0 \} } e^{\frac{2 \pi i a j}{p}} -1.
$$
\end{proof}

%\begin{cor}\label{cor_gauss_val}
%If $p$ is an odd number of the form $p=q^k$, where $k\geq 1$ is an integer and $q$ is prime, then for each $a\in \mathbb Z_p^*$, we have
%$$
%\sum\limits_{j \in \RES{2} \cup \{0 \} } e^{\frac{2 \pi i a j}{p}} = 
%\left\{
%\begin{array}{cc}
%\frac{l(r+1) + q^{k/2} -p}{l}
%,  &  k  \text{ even}\\
%\frac{l(r+1) + q^{k/2} (a/q)_2 -p}{l}
% & k \text{ odd, } q \equiv_4 1\\
%\frac{l(r+1) + q^{k/2} (a/q)_2 i -p}{l}
%, & k \text{ odd, } q \equiv_4 3\\
%\end{array}
%\right.,
%$$
%where $r = \left|\RES{2}\right|$ 
%and 
%$l= \left[\mathbb Z_p^* : \RES{2} \right]$ 
%denotes the index of $\RES{2}$ in $\mathbb Z_p^*$.
%\end{cor}
%\begin{proof}
%This follows from the preceding theorem and the observation that
%$$
%\sum\limits_{x \in \mathbb Z_p} e^{\frac{2 \pi i a x^2}{p}}
%=
%\left(p - l(r+1) \right) + l \sum\limits_{\tiny j \in \RES{2} \cup \{0 \} } e^{\frac{2 \pi i a j}{p}}.
%$$
%\end{proof}

Now we show that every harmonic frame generated by a Gaussian difference set is either equiangular or biangular.

\begin{thm}
Let $\mathbb G = \mathbb Z_p^*$, where $p$ is an odd prime, and suppose that $\mathcal S \subset \mathbb G$ is a $(p, m,  \lambda, \mu )$-Gaussian difference set for $\mathbb G$.  If $\mathcal F = \{ f_g \}_{g \in \mathbb G}$ is the harmonic frame for $\mathbb C^m$ generated by $\mathcal S$, then either
\begin{enumerate}
\item $p \equiv_4 3$, $\mathcal S$ is a $(p, m, \lambda)$-difference set  for $\mathbb G$ and $\mathcal F$ is an equiangular tight frame, or
\item  $p \equiv_4 1$,  $\mathcal S$ is a $(p, m, \lambda)$-difference set  for $\mathbb G$ and $\mathcal F$ is an equiangular tight frame, or
\item   $p \equiv_4 1$ and $\mathcal F$ is a biangular tight frame with frame angles $\alpha_1, \alpha_2$ and frame angle multiplicities $\tau_1, \tau_2$,
where
$$
\alpha_1, \alpha_2 =
 \frac{1}{m}
\sqrt{
 (m-\lambda) + (\lambda - \mu) \left(  \frac{1 \pm \sqrt{p}}{2}\right)  
}$$
and $\tau_1 = \tau_2 = \frac{p-1}{2}$.
\end{enumerate}
\end{thm}

\begin{proof}
If $p \equiv_4 3$, then Proposition~\ref{prop_angles_simp2} and Corollary~\ref{cor_gauss_val} show that 
$$
m^2 | \langle f_x, f_y \rangle |^2 = (m-\lambda) + (\lambda - \mu) \left( \frac{1 \pm \sqrt{p} i}{2} \right)
$$ for every $x,y \in \mathbb G$ with $x \neq y$, and since this value is always a nonnegative real number, it follows that $\lambda=\mu$, which means $\mathcal S$ is a $(p,m,\lambda)$-difference set for $\mathbb G$ and therefore $\mathcal F$ is equiangular by Theorem~\ref{thm_char_ETFs}.

If $p \equiv_4 1$, then either $\lambda=\mu$ or $\lambda \neq \mu$.   If $\lambda = \mu$, then $\mathcal S$ is a $(p,m,\lambda)$-difference set for $\mathbb G$, so $\mathcal F$ is equiangular  by Theorem~\ref{thm_char_ETFs}.  If $\lambda \neq \mu$, then Corollary~\ref{cor_char_kangular} together with Corollary~\ref{cor_gauss_val} show that $\mathcal F$ is biangular.  In this case, the claimed values for the frame angles are computed by substituting the value of the summation from Corollary~\ref{cor_gauss_val} into Equation~(\ref{eq_angle_simp2}) from Proposition~\ref{prop_angles_simp2} and the claimed values for the frame angle multiplicites follow from the fact $\left|\RES{2} \right| =  \frac{p-1}{2}$. 
\end{proof}

We conclude this subsection by constructing and studying a family of Gaussian difference sets which admit examples which are neither divisible difference sets nor partial difference sets. In particular,  the {quartic residues} in certain cyclic groups of odd prime order have this property.  In order to do this, we require a few more results from classical number theory.

Given an odd prime $p$, we define the
{\bf quartic residue symbol} by
$$
(a/p)_4 \equiv_p a^{\frac{p-1}{4}}.
$$
\begin{prop}\label{prop_2notres}
If $p$ is an odd prime and
 $a \in \RES{2}$, then
$$(a/p)_4 \equiv_p \left\{ \begin{array}{cc} 1, & a \in \RES{4} \\ -1 , & a \notin \RES{4} \end{array} \right. .
$$
\end{prop}
\begin{proof}
Since $a \in \RES{2}$, we can write $a\equiv_p b^2$ for some $b \in \mathbb Z_p^*$.
 If $a\notin \RES{4}$, then $b\notin \RES{2}$, so we have $(a/p)_4 \equiv_p b^{\frac{p-1}{2}}\equiv_p(b/p)_2\equiv_p -1.$  Otherwise, we can write $a\equiv_p c^4$ for some $c  \in \mathbb Z_p^*$, and we have 
 $(a/p)_4 \equiv_p (c^2)^{\frac{p-1}{2}}\equiv_p (c^2/p)_2\equiv_p 1.$
\end{proof}

The following characterization of whether $2$ is a quadratic residue is due to Gauss \cite{MR1625181}.
\begin{thm}\label{th_gauss1}[Gauss {\cite{MR1625181}}]
 Given a prime $p$,
then
$$
(2/p)_2
\equiv_p
 \left\{ \begin{array}{cc} 
1 , & p\equiv_8 \pm 1  \\ 
-1 , & p \equiv_8 \pm 3 \end{array} \right..
$$
\end{thm}
%Since $\RES{4} \subset \RES{2}$, this implies the following corollary.
%\begin{cor}\label{cor_2_nonres}
%If $p$ is prime and $p \equiv_8 \pm 3$, then
% $2 \notin \RES{4}$.
%\end{cor}

Whenever $p\equiv_4 1$, Gauss also showed that $\RES{4}$ has four distinct cosets in $\mathbb Z_p^*$.  We denote the coset of $\RES{4}$ in $\mathbb Z_p^*$ with representative $a \in \mathbb Z_p^*$ by
$$a \RES{4} := \left\{a r :  r \in \RES{4} \right\}.$$
\begin{thm}\label{th_gauss2}[Gauss, \jh{\cite{MR1625181}}]
Let $p$ be prime with $p \equiv_4 1$.  If $a \in \mathbb Z_p^*$ with $(a/p)_2 =_p -1$, then $\mathbb Z_p^*$ can be written as the disjoint union 
$$\mathbb Z_p^*
=
\RES{4} \dot \cup a \RES{4} \dot \cup a^2 \RES{4} \dot \cup a^3 \RES{4}.
$$ 
\end{thm}

Now we show how to form Gaussian difference sets from the quartic residues in certain groups of prime order.

%Finally, we show that the quartic residues for our special case produce harmonic BTFs.
%\begin{proof}[Proof of Theorem~\ref{th_quartic_btfs2}]
%%By Theorem~\ref{th_quartic_btfs}, $\RES{4}$ is a  $(p, \frac{p-1}{4}, \frac{p+1}{2}, \lambda, \mu)$-bidifference set for $\mathbb Z_p$ relative to $\RES{2} \cup \{0\}$, where $\lambda + \mu=q$. 
%If $\lambda=\mu$, then $\RES{4}$ is an $(p,\frac{p-1}{4},\lambda)$-difference set for $\mathbb G$, so, by Theorem~\ref{thm_char_ETFs}, $\mathcal F$ is equiangular. If $\lambda \neq \mu$ and
% and $x, y \in \mathbb Z_p$ with $z = y -x \neq 0$, then it follows from Proposition~\ref{prop_angles_simp2} that
%\begin{align*}
%\left(\frac{4}{p-1}\right)^2 |\langle f_x, f_y \rangle |^2
%&=
%(m-\lambda)+ (\lambda - \mu) \sum\limits_{\xi \in \RES{2}\cup\{0\} }\rho_{z} (\xi)
%\\
%&=
%(m-\lambda)+ (\lambda - \mu) \sum\limits_{\xi \in \RES{2}\cup\{0\} }e^{\frac{2 \pi i z \xi}{p}}. 
%\end{align*}
%By substituting the value for the quadratic Gauss sum from Theorem~\ref{th_gsum}, it follows that $\mathcal F$ is biangular with the claimed frame angles, and the claimed frame angle multiplicites follow from Observation~\ref{obs_quads_part} from the proof of Theorem~\ref{th_quartic_btfs}.
%\end{proof}

\begin{thm}\label{th_quartic_btfs}
Let $\mathbb G = \mathbb Z_p$.  If $p$
is a prime of the form $p=8q+5$, where $q>0$,
then there exist nonnegative integers $\lambda$ and $\mu$, , where $\lambda + \mu =q$, such that
$\RES{4}$ is a $(p,\frac{p-1}{4}, \lambda , \mu)$-Gaussian difference set for $\mathbb Z_p$ and $\RES{4} \cup \{0\}$ is a $(p,\frac{p+3}{4}, \lambda+1, \mu)$-Gaussian difference set for $\mathbb Z_p$.
\end{thm}

\begin{proof}
Since $p=8q+5$, it follows from   Theorem~\ref{th_gauss1} that $(2/p)_2 \equiv_p ~-1 $, so Theorem~\ref{th_gauss2}  gives the disjoint union
$$\mathbb Z_p^*
=
\RES{4} \dot \cup 2 \RES{4} \dot \cup 4 \RES{4} \dot \cup 8 \RES{4},
$$
and it follows from the definition of the quadratic residue symbol  that
$$
 \left(2^j/p\right)_2 \equiv_p ~(-1)^j \text{ for every } j \in \mathbb Z.
$$
In particular, this computation and Euler's criterion show that
 \begin{equation}\label{obs_quads_part}
%a \in \RES{2} \text{ if and only if } a \in \RES{4} \dot \cup 4 \RES{4}.
\RES{2} =  \RES{4} \dot \cup 4 \RES{4}.
\end{equation}

Next,   we show that $-1 \in 4 \RES{4}$ and $-2 \in 8 \RES{4}$.  Observe that
$$
 (-1 / p)_2 \equiv_p (-1)^{\frac{p-1}{2}} \equiv_p (-1)^{4q +2} \equiv_p 1, 
$$
so $-1 \in \RES{2}$ by Euler's criterion, but
$$
 (-1 / p)_4 \equiv_p (-1)^{\frac{p-1}{4}} \equiv_p (-1)^{2q +1} \equiv_p -1, 
$$
so $-1 \notin \RES{4}$ by Proposition~\ref{prop_2notres}.  Thus, Observation~(\ref{obs_quads_part}) implies that
$
-1 \in 4 \RES{4}.
$
Since $-1 \equiv_4 4 c^4$ for some $c \in \mathbb Z_p^*$, we also conclude that
$-2 \in 8 \RES{4}$, since
$-2 \equiv_p (-1)(2) \equiv_p (4 c^4)(2) \equiv_p 8 c^4$.

Next, we consider the differences between the quartic residues.
For each $a \in \mathbb Z_p^*$, let 
$$
D(a) = \left\{(x,y) \in \mathbb Z_p^* \times  \mathbb Z_p^* :
 x^4 -y^4 \equiv_p a \right\} .
$$
Fix $j \in \{0, 1,2,3\}$ and let $a, a' \in 2^j \RES{4}$, where $a=z^4 2^j$ and $a'=(z')^4 2^j$.  
If $(x,y), (x',y')\in D(a)$, then $$x^4 - y^4 \equiv_p z^4 2^j 
\text{ implies } ({z'} z^{-1} x)^4 -(z' z^{-1}y)^4 \equiv_p (z')^4 2^j,$$ so $({z'} z^{-1} x,{z'} z^{-1} y) \in D(a')$.  The group structure of $\mathbb Z_p^*$ implies 
$$(z' z^{-1} x, z' z^{-1} y)  = (z' z^{-1}x', z' z^{-1}y')   \text{ if and only if } (x,y) = (x',y'),$$
so it follows that $|D(a)| \leq |D(a')|$, and since this argument is symmetric with respect to $a$ and $a'$, we conclude that $|D(a)| = |D(a')|$ for all $a,a' \in 2^j \RES{4}$.

Observe that 
$(x,y) \in D(1)$ 
if and only if
$(y,x) \in D(-1)$, and since $-1 \in 4 \RES{4}$, it follows that
$$
|D(a)| = |D(1)| = |D(-1)| = |D(a')|
$$
for all $a \in \RES{4}$ and $a' \in 4 \RES{4}$.
  Similarly,
$(x,y) \in D(2)$ 
if and only if
$(y,x) \in D(-2)$, and since $-2 \in 8 \RES{4}$, it follows that
$$
|D(a)| = |D(2)| = |D(-2)|  =  |D(a')|
$$
for all $a \in 2 \RES{4}$ and $a' \in 8 \RES{4}$.

Since each coset $2^j \RES{4}$ is of size $\left|2^j \RES{4} \right| = \frac{p-1}{4}$, it is evident  that $\RES{4}$ forms 
a $(p, \frac{p-1}{4}, \lambda, \mu)$-Gaussian difference set for $\mathbb Z_p$, where
$\lambda = |D(1)|$ and $\mu = |D(2)|$. The fact that $\lambda +  \mu=q$ follows from elementary counting and the bidifference set condition, 
 $$
\left|
\{ (a,b) \in \RES{4} \times \RES{4} : a \neq b
\}
\right|
=
\lambda  \left|   \RES{2}  \right| + \mu \left|  \mathbb Z_p^* \backslash  \RES{2} \right|.
$$

Since the multiset of pairwise differences between the elements of $\RES{4}\cup \{0\}$ is that of $\RES{4}$ augmented with the possible differences with respect to $0$, the claim that $\RES{4}\cup \{0\}$ is a $(p,\frac{p+3}{4}, \lambda+1, \mu)$-Gaussian difference set for $\mathbb Z_p$ follows because
$$
\left\{ a-0 : a\in\RES{4}  \right\} \cup \left\{ 0-a : a\in\RES{4} \right\} 
=
\RES{4} \cup 4\RES{4} = \RES{2}.
$$
\end{proof}
\begin{rem}
Under the same conditions as this theorem, a straightforward adjustment of the proof shows that $2\RES{4}, 4\RES{4}$, and  $8\RES{4}$ and, similarly, $2\RES{4}\cup\{0\}, 4\RES{4}\cup\{0\}$, and  $8\RES{4} \cup \{0\}$ form Gaussian difference sets for $\mathbb G$ that generate harmonic frames with the same frame angle sets as the biangular frames generated by $\RES{4}$ and $\RES{4}\cup\{0\}$, respectively.
\end{rem}
As a corollary, we observe that there exist Gaussian difference sets which are neither divisible difference sets nor partial difference sets.

\begin{cor}
Let $\mathbb G = \mathbb Z_p$.  If $p$
is a prime of the form $p=8q+5$, where $q>0$ and $q$ is odd, then $\RES{4}$ is a proper $(p,\frac{p-1}{4},\frac{p+1}{2}, \lambda, \mu)$-bidifference set for $\mathbb G$ relative to $\RES{2} \cup \{0\}$.
% and $\RES{4} \cup \{0\}$ is a proper $(p,m,\frac{p+3}{4}, \lambda+1, \mu)$-%bidifference set $\mathbb G$ relative to $\RES{2} \cup \{o\}$. 
 Moreover, $\RES{4}$
%and $\RES{4} \cup \{0\}$
forms a Gaussian difference set for $\mathbb G$, but %neither of them forms a 
it does not form a partial difference set for $\mathbb G$ and 
%neither of them forms a 
it does not form a divisible difference set for $\mathbb G$.
\end{cor}

\begin{proof}
By the preceding theorem, the Gaussian difference set $\RES{4}$ is a proper $(p,\frac{p-1}{4},\frac{p+1}{2}, \lambda, \mu)$-bidifference set for $\mathbb G$ relative to $\RES{2} \cup \{0\}$, where $\lambda \neq \mu$ because $q$ is odd.
Since $\mathbb G$ is a group of prime order, it cannot have a subgroup of order $\frac{p+1}{2}$, so $\RES{4}$ cannot be a divisible difference set for $\mathbb G$. \jh{Furthermore, $\RES{4}$ cannot be a partial difference set for $\mathbb G$ because $\left| \RES{4}\right|=\frac{p-1}{4} <\frac{p-1}{2}= \left| \RES{2} \right|$.}
\end{proof}

Some of the bidifference sets obtained from quartic residues in Theorem~\ref{th_quartic_btfs} have already been considered \cite{MR2246267, 2014arXiv1409.0114N}.  In fact, some of them are difference sets \cite{MR2246267}.

\begin{thm}[\cite{MR2246267}]\label{th_quart_diff_sets}
Let $\mathbb G = \mathbb Z_p$, where $p$ is prime.  If  $p=4a^2 +1$ with $a$ an odd integer, then $\RES{4}$ forms a $(p,\frac{p-1}{4},\frac{p-5}{16})$-difference set for $\mathbb G$, and if $p=4a^2 +9$ with $a$ an odd integer, then  $\RES{4} \cup \{ 0 \}$ forms a $(p,\frac{p+3}{4},\frac{p+3}{16})$-difference set for $\mathbb G$. 
\end{thm}

Another case where the bidifference sets obtained in Theorem~\ref{th_quartic_btfs} have been considered is when they form {\it almost difference sets} \cite{2014arXiv1409.0114N}.

\begin{defn}
Let $\mathcal S \subset \mathbb G$.   If $\mathcal S$ is an $(n,m,l, \lambda, \lambda+1)$-bidifference set relative to some subset $A \subset \mathbb G$, then it is called an {\bf $(n,m, \lambda, l-1)$-almost difference set for $\mathbb G$}. 
\end{defn}
\begin{rem}
Some authors define almost difference sets as above with the additional constraint that they form divisible difference sets \cite{MR1440858}.  
\end{rem}
Almost difference sets are interesting combinatorial objects  with useful applications;  see \cite{2014arXiv1409.0114N} for a survey.
\jh{As can be verified in the tables of the appendix, there are numerous examples of almost difference sets that manifest as divisible difference sets and partial difference sets} and therefore generate harmonic BTFs; however, as is demonstrated in Example~\ref{ex_bds_counter}, there are examples of almost difference sets which do not generate biangular harmonic frames, so we do not conduct an extensive study of them here.  Our main interest in this class of bidifference sets is the relationship they have with the bidifference sets obtained via quartic residues from Theorem~\ref{th_quartic_btfs}.  The following results are shown in \cite{2014arXiv1409.0114N} using the theory of cyclotomic numbers.

\begin{thm}\label{th_almost_quar}[Section 6 of \cite{2014arXiv1409.0114N}; see also \cite{MR2070151, MR1725160}]
Let $\mathbb G = \mathbb Z_p$, where $p$
is an odd prime,
then the following statements hold.
\begin{enumerate}
\item
If $p=9 + 4a^2$ or $p=25 + 4a^2$ for some integer $a$, then
$\RES{4}$ is a $(p, \frac{p-1}{4}, \frac{p-13}{16}, \frac{p-1}{2})$-almost difference set for $\mathbb G$.
\item
If $p=1 + 4a^2$ or $p=49 + 4a^2$ for some integer $a$, then
$\RES{4}\cup\{0\}$ is a $(p, \frac{p+3}{4}, \frac{p-5}{16}, \frac{p-1}{2})$-almost difference set for $\mathbb G$.
\end{enumerate}
\end{thm}

Whenever the  Gaussian difference sets obtained as quartic residues from Theorem~\ref{th_quartic_btfs} coincide with the difference sets from Theorem~\ref{th_quart_diff_sets} or with the almost difference sets from Theorem~\ref{th_almost_quar}, then we can state more precisely the angles sets of the corresponding harmonic frames.

\begin{cor}
Let $\mathbb G = \mathbb Z_p$, where $p$
is a prime of the form $p=8q+5$ with $q>0$, let $\mathcal S \subset \mathbb G$ be nonempty, and let $\mathcal F$ be the $\mathbb G$-harmonic frame for $\mathbb C^m$ generated by $\mathbb G$.
\begin{enumerate}
\item
If $\mathcal S=\RES{4}$ and $p=4a^2+1$, where $a$ is an odd integer, then $\mathcal F$ is an equiangular tight frame.
\item
If $\mathcal S=\RES{4}\cup\{0\}$ and $p=4a^2+9$, where $a$ is an odd integer, then $\mathcal F$ is an equiangular tight frame.
\item
If $\mathcal S=\RES{4}$ and either $p=9+4a^2$ or $p=25+4a^2$, where $a$ is an integer, then $\mathcal F$ is a biangular tight frame with frame angles
$$
\alpha_1, \alpha_2 = 
 \frac{1}{p-1}
\sqrt{
3p + 1 \pm8\sqrt{p}
      }
$$ and frame angle multiplicities $\tau_1 = \tau_2 = \frac{p-1}{2}$.
\item
If $\mathcal S=\RES{4} \cup \{0\}$ and either $p=1+4a^2$ or $p=49+4a^2$, where $a$ is an integer, then $\mathcal F$ is a biangular tight frame with frame angles
$$
\alpha_1, \alpha_2 = 
 \frac{1}{p+3}
\sqrt{
 3p + 9 \pm 8 \sqrt{p}
}
$$ and frame angle multiplicities $\tau_1 = \tau_2 = \frac{p-1}{2}$.
\end{enumerate}  
\end{cor}

%
%%
%%Given a generator $z \in \mathbb Z_p^*$, meaning that $\mathbb Z_p^* = \{ z^j : j \in \mathbb Z \}$, it follows that the function
%% $$\Psi: \mathbb Z \rightarrow \mathbb Z_p^* / \RES{4}, 
%%j \mapsto  [z^j]$$
%%is a surjective homomorphism from the additive group $\mathbb Z$ to the quotient group  $\mathbb Z_p^* / \RES{4}$
%%%we may form the factor
%%%Given an odd prime $p$ with $p \equiv 5 \mod 8$,
%%%% the function $$\Psi: \mathbb Z_p^* \rightarrow \RES{4}, 
%%%%a \mapsto  a^4
%%%%$$
%%%%is clearly a surjective homomorphism, so $\mathbb Z_p^* / \ker{\Psi} \cong \RES{4}$.
%%%then
%%%since
%%% $2 \notin \RES{4}$ by Corollary~\ref{cor_2_nonres}, we have
%%%$$
%%%\mathbb Z_p^* / \RES{4} = \left\{ a \RES{4} : s\in \{0,1,2,3\} \right\}
%%%%$$
%%%the elements of are of the form
%%%$$
%%% \left[z^j\right]=z^j \RES{4} = \{z^j r^4 :  j \in \mathbb Z \}.
%%%$$ 
%%% T
%%
%% 
%
%
%

\subsection{Nested Divisible Difference Sets}
We conclude our work with an examination of a third class of nested difference sets, the {\it nested divisible difference sets}, which can be viewed as a generalization of the divisible difference sets.  This class of combinatorial structures admits an example of a harmonic biangular frame which is not generated by a bidifference set.

\begin{defn}
Let $\mathcal S$ be a $(n,m,t)$-nested difference set for $\mathbb G$ relative to $(\mathcal A, \Lambda)$ .  We say that $\mathcal S$ is an {\bf $(n,m,t)$-nested divisible difference set for $\mathbb G$ relative to $(\mathcal A, \Lambda)$}
if every element  $A \in \mathcal A$ forms a subgroup of $\mathbb G$.
\end{defn}

%	In general, Harmonic frame generated by nested divisible difference sets are not biangular. We have an counterexample as below.
%	\begin{ex}
%		Let $S=\{1, 3, 4, 5, 7, 8, 9, 11\}\subset \mathbb{Z}_{12}$. Then $S$ is a $(12, 8, 3)$ nested divisible difference set for $\mathbb{Z}_{12}$ relative to $(\mathcal A, \Lambda)$, where $\mathcal A=\{A_0=\{0\}, A_1=\{0, 4, 8\}, A_2=\{0, 2, 4, 6, 8, 10\}, A_3=\mathbb{Z}_{12}\}$ and $\Lambda=\{\lambda_1=7, \lambda_2=6,\lambda_3=4\}$. The Harmonic frame generated by $S$ has three angles: $\dfrac{1}{2}, \dfrac{1}{4}$ and $\dfrac{1}{8}$.
%	\end{ex}
	
	Next, we characterize the biangular harmonic frames  generated by nested divisible difference sets.

	\begin{theorem}\label{thm_btfs_ndds_char}
		Let $\mathcal S= \{g_1,...,g_m \} \subset \mathbb G$ be a proper
		$(n,m,t)$-nested divisible difference set for $\mathbb G$ relative to $(\mathcal A, \Lambda)$, where $t\geq 2$, $\mathcal A = \{A_j\}_{j=0}^t$ and  $\Lambda = \{\lambda_j\}_{j=1}^t$, and let $s\in\{1,...,t-1\}$ be the smallest integer such that $\lambda_s\not=\lambda_{s+1}$. 
Furthermore, let
$$
\alpha_1 = \sqrt{\frac{1}{m} - \frac{\lambda_1}{m^2}}
\text{ and }
\alpha_2 =\sqrt{\alpha_1^2 + \dfrac{1}{m^2}(\lambda_{s}-\lambda_{s+1})|A_{s}|}
.$$
 If $\mathcal F$ is the $\mathbb G$-harmonic frame generated by $\mathcal S$,
then 
\begin{enumerate}
\item
$\alpha_1$ and $\alpha_2$ occur among the frame angles of $\mathcal F$, and
\item
$\mathcal F$ is biangular if and only if, for any $r \in \{ s+1, ..., t-1\}$, either
		$$\displaystyle\sum\limits_{j=s}^r(\lambda_j-\lambda_{j+1})|A_j|=0\quad \text{ or }\quad
		\displaystyle\sum\limits_{j=s+1}^r(\lambda_j-\lambda_{j+1})|A_j|=0.$$ 
\end{enumerate}
%		where .
	\end{theorem}
	
	\begin{proof}
By Lemma~\ref{prop_mod_entries}, the squared Hilbert Schmidt norm of the $\xi$-th modulation operator is
$$
\| X_\xi \|_{H.S.}^2 = 
\left\{
\begin{array}{cc}
\frac{n^2}{m}, & \xi=0 \\
\lambda_j \frac{n^2}{m^2}, & \xi \in A_j \backslash A_{j-1}
\end{array}
\right.
.
$$

		Given $x, y \in \mathbb G$ with $z=y-x \neq 0$, then substituting these values into Equation~(\ref{eq_encode_angles_simp}) yields
		\begin{align*}  |\langle f_x, f_y \rangle |^2 &=\dfrac{1}{m}+
		\dfrac{1}{m^2}\sum_{j=1}^t 
		\lambda_j\sum\limits_{\xi \in A_j \backslash A_{j-1}}
		\rho_{z} (\xi)\\
		&=\dfrac{1}{m}-\dfrac{\lambda_1}{m^2}+\dfrac{\lambda_1}{m^2} 
		\sum\limits_{\xi \in A_1}
		\rho_{z} (\xi)
		+
		\dfrac{1}{m^2}\sum_{j=2}^t\lambda_j
		\sum\limits_{\xi \in A_j \backslash A_{j-1}}
		\rho_{z} (\xi)\\
		&=\dfrac{1}{m}-\dfrac{\lambda_1}{m^2}+\dfrac{1}{m^2}\sum_{j=1}^{t-1}(\lambda_j-\lambda_{j+1}) \sum\limits_{\xi \in A_j}
		\rho_{z} (\xi)
+ \sum\limits_{\xi \in A_t}
		\rho_{z} (\xi).
		\end{align*}
Since $A_t = \mathbb G$, the last term vanishes by Corollary~\ref{cor_gen_sum_roots_unity}, and if $s\in\{1,...,t-1\}$ is the smallest integer such that $\lambda_s\not=\lambda_{s+1}$, then this simplifies to
\begin{equation}\label{eq_ndd_1}
|\langle f_x, f_y \rangle |^2
=
\dfrac{1}{m}-\dfrac{\lambda_1}{m^2}+\dfrac{1}{m^2}\sum_{j=s}^{t-1}(\lambda_j-\lambda_{j+1}) \sum\limits_{\xi \in A_j}
		\rho_{z} (\xi)
.
\end{equation}
	%	
% For each $j \in \{0,1, ..., t-1\}$, 
%the fact that each $A_j$ is a subgroup of $\mathbb G$ and $A_j\subset A_{j+1}$ implies that
Observe that the increasing subgroup structure of $\mathcal A$ implies that the structure of the corresponding annihilator subgroups is decreasing; that is, 
 $$\AN{(A_{j'})} \subset\AN{(A_{j})} \text{ for all }0 \leq j\leq j' \leq t.$$ 
Using this observation and Proposition~\ref{prop_sumchar_on_H}, we re-express Equation~(\ref{eq_ndd_1}) as
\begin{equation}\label{eq_nest_div_btf}
|\langle f_x, f_y \rangle |^2
=
\left\{
{\tiny
\begin{array}{cc}
\alpha_1^2, & 
\begin{array}{cc}
\rho_z \in \AN{(A_r)} \backslash \AN{(A_{r+1})} \\
 0 \leq r < s
\end{array}
 \\
\alpha_2^2,
&  
\rho_z \in \AN{(A_s)} \backslash \AN{(A_{s+1})} 
\\
\alpha_2^2+\dfrac{1}{m^2}\sum_{j=s+1}^r(\lambda_{j}-\lambda_{j+1})|A_{j}|,
&  
\begin{array}{cc}
\rho_z \in \AN{(A_r)} \backslash \AN{(A_{r+1})} \\
 s < r \leq t-1
\end{array}
\end{array}
}
\right.
,\end{equation}
where
$$
\alpha_1 = \sqrt{\frac{1}{m} - \frac{\lambda_1}{m^2}}
\text{ and }
\alpha_2 =\sqrt{\alpha_1^2 + \dfrac{1}{m^2}(\lambda_{s}-\lambda_{s+1})|A_{s}|}
.$$

Because $\mathcal S$ is proper, it follows that
 $A_{j+1} \backslash A_j$ is nonempty, or equivalently $|A_j| < |A_{j+1}|$, for every $j \in \{0,1,...,t-1\}$;  therefore, Proposition~\ref{prop_size_ann} implies that 
 $\AN{(A_j)} \backslash \AN{(A_{j+1})}$ is nonempty for every $j \in \{0,1,...,t-1\}$.  

In particular, $\AN{(A_0)} \backslash \AN{(A_{1})}$ and $\AN{(A_s)} \backslash \AN{(A_{s+1})}$ are both nonempty, which shows that $\alpha_1$ and $\alpha_2$ must occur among the frame angles of $\mathcal F$.  The claim follows by checking the conditions for which $|\langle f_x, f_y \rangle| \in \{\alpha_1, \alpha_2 \}$ in the case $s<r\leq t-1$ from Equation~(\ref{eq_nest_div_btf}).

	\end{proof}
	
By the preceding theorem, we see that the conditions under which a proper $(n,m,t)$-nested divisible difference set generates a biangular harmonic frame are quite restrictive if $t \geq 3$.  Nevertheless, the following example demonstrates the existence of a proper $(8,3,3)$-nested divisible difference set that generates a biangular harmonic frame.  

\begin{ex}\label{ex_ndds_btf}
The $(8,3,3)$-nested difference set from Example~\ref{ex_nest_div_833} is a proper  $(8,3,3)$-nested divisible difference set because $A_1$ and $A_2$ are both subgroups of $\mathbb G$.  Furthermore, this example satisfies the conditions of Theorem~\ref{thm_btfs_ndds_char} and therefore generates a biangular harmonic frame.
%
%; therefore, the $\mathbb G$-harmonic frame generated by $\mathcal S$ is an equiangular tight frame for $\mathbb C^3$ consisting of $7$ vectors.  
\end{ex}

A natural question to ask about the biangular frame constructed in Example~\ref{ex_ndds_btf} is whether it is possible to construct a harmonic frame for $\mathbb C^3$ consisting of $8$ vectors with the same angle set using a bidifference set instead. 
 Up to isomorphism, there are three $8$-element groups to consider, $\mathbb G = \mathbb Z_8$, $\mathbb G =\mathbb Z_2 \oplus \mathbb Z_4$ or
$\mathbb G =\mathbb Z_2 \oplus\mathbb Z_2 \oplus \mathbb Z_2$, and, for each of these, there are $56$ ways to select a $3$-element subset $\mathcal S \subset \mathbb G$.  We inspected all possibilities for these three groups and concluded the following.  If $\mathbb G =\mathbb Z_2 \oplus\mathbb Z_2 \oplus \mathbb Z_2$, then there are no $3$-element subsets $\mathcal S \subset \mathbb G$ such that $\mathcal S$ generates a $\mathbb G$-harmonic frame with the same frame angle set as the frame in  Example~\ref{ex_ndds_btf}.  If $\mathbb G =\mathbb Z_2 \oplus\mathbb Z_4$, then there are $32$ subsets $\mathcal S \subset \mathbb G$ such that $\mathcal S$ generates a $\mathbb G$-harmonic frame with the same frame angle set as the frame in  Example~\ref{ex_ndds_btf} and they are all  proper  $(8,3,3)$-nested divisible difference sets.  If $\mathbb G =\mathbb Z_8$, then there are $16$ subsets $\mathcal S \subset \mathbb G$ such that $\mathcal S$ generates a $\mathbb G$-harmonic frame with the same frame angle set as the frame in  Example~\ref{ex_ndds_btf} and they are all  proper  $(8,3,3)$-nested divisible difference sets.  Thus, by means of proof by exhaustion, we conclude the following.
\begin{thm}
There exists a group $\mathbb G$ of order $n$ and a biangular $\mathbb G$-harmonic frame $\mathcal F$ for $\mathbb C^m$ such that, for every group $\mathbb G'$ of order $n$, there is no bidifference set for $\mathbb G'$ that generates a $\mathbb G'$-harmonic frame for $\mathbb C^m$ with the same set of frame angles as $\mathcal F$.
\end{thm}

\appendix

\section{Tables}
In the following tables, we collect several results regarding the existence  of infinite families of proper divisible difference sets, relative difference sets, and partial difference sets for abelian groups and compute information about the frame angle sets of the corresponding harmonic frames.  The purpose of these tables is not  to provide an exhaustive list of all possible biangular harmonic frames generated by bidifference sets, but to demonstrate the applicability of the theory developed in Section~\ref{sec_btfs}.

In each table, there are four columns.  For a given row from Table~\ref{tbl_dds}, Table~\ref{tbl_rds}, or Table~\ref{tbl_pds}, the entry of the first column describes the sufficient conditions for the existence of a divisible difference set, relative difference set or partial difference set, respectively, with parameters given in the entry from the second column of the same row, and a reference for the result is listed in the fourth column.  In each entry from the third column of a given table, we compute the frame angles, $\alpha_1$ and $\alpha_2$, of the harmonic frame generated by the bidifference set of that row.
%, and, in the fourth column, we compute the corresponding frame multiplicities, $\tau_1$ and $\tau_2$. 

Recall that $\mathbb G$ is an abelian group of order $n$ and that a $\mathbb G$-harmonic frame $\mathcal F$ consisting of $n$ vectors for $\mathbb C^m$ is defined by how we select a subset of $m$ characters from $\widehat{\mathbb G}$, as described in Section~\ref{sec_mod}.
For the second column of Table~\ref{tbl_dds}, Table~\ref{tbl_rds}, or Table~\ref{tbl_pds}, we use the parametrization of divisible difference sets, relative difference sets, or partial difference sets as given in Definition~\ref{def_dds}, Definition~\ref{def_rds}, and Definition~\ref{def_pds}, respectively, and we compute the information about the corresponding frame angles with Theorem~\ref{thm_dds_btf}, Corollary~\ref{cor_rds_btf} and Theorem~\ref{cor_pds_btf}, respectively.

Because our primary interest is in the existence of the resulting biangular harmonic frames, we list the sufficiency conditions for the existence of the underlying bidifference sets without describing the details of their constructions. For typographical reasons, we symbolize frequently occurring and complicated expressions occurring among the conditions with lower-cased Roman numbers, as listed in Table~\ref{tbl_key}.

\begin{table}[htb]
\centering
\caption{Common expressions in harmonic  BTF tables}\label{tbl_key}
\begin{tabular}{|P{.95\linewidth}|}
\hline
%{\large  \underline{Common conditions in harmonic  BTF tables} }
		\begin{enumerate}[label=(\roman*)]
		\item\label{cond_prime} 
		``$p$ is a prime"
		
		\item\label{cond_mersenne} 
		``$p$ is a prime of the form $p=2^s -1$  for some $s\in\mathbb N$ (ie,  $p$ is a Mersenne prime)"
		
		\item\label{cond_primepower} 
		``$q=p^s$ for some prime $p$ and $s\in\mathbb N$ (ie, $q$ is a prime power)"
		
		\item\label{cond_hadamard}
		 ``$u \in \mathbb N$, there exists an abelian group $\mathbb G'$ of order $4u^2$, and either there exists a $\left(4u^2, 2u^2+u, u^2+u \right)$-difference for $\mathbb G'$ or  there exists a $\left(4u^2, 2u^2-u, u^2-u \right)$-difference for $\mathbb G'$"
				
\item\label{cond_diffset1}
		 ``$v,w\in\mathbb N$, there exists an abelian group $\mathbb G''$ of order $w$, and there exists a 
		$\scriptscriptstyle \left(
		w,v, \frac{v(v-1)}{w-1}
		\right)$-difference set for $\mathbb G''$"

\item \label{cond_subgroup_EAq}
		 ``$\mathbb G$ contains a subgroup isomorphic to $\mathbb Z_p^{as}$"

\item\label{cond_prime_decomp}
``$\scriptscriptstyle v \in \mathbb N$,
$a_1,...,a_v \in \mathbb N$, and
$p_1,...,p_v $ are distinct primes''
		\end{enumerate}
\\
\hline
	     \end{tabular}
\end{table}

\begin{remark}
With respect to condition~\ref{cond_diffset1}, numerous examples of difference sets are known.  We refer to \cite{MR1440858, MR2246267} for more details.
\end{remark}
\begin{remark}
A difference set with parameters described in~\ref{cond_hadamard} is called a {\it Hadamard difference set}.  Numerous constructions of these are known.  We refer to \cite{MR1440858, MR1400415, MR2246267} for more details.
\end{remark}

We remark once more that our notation for divisible difference sets deviates from the form typically seen in the literature \cite{MR1440858}; see also Remark~\ref{rem_DDS_RDS_notn}.

%\clearpage

%%%%%%%%%%%%%%%%%%%%%%%%%%%%%%%%%%%%%%%%%%%%%%%%%%%%%%%%%%%%%%%%%%%%%%%%%%%%%%%%%%%%%%%%%%%%%%
%%%%%%%%%DDS TABLE%%%%%%%%%%%%%%%%%%%%%%%%%%%%%%%%
%%%%%%%%%%%%%%%%%%%%%%%%%%%%%%%%%%%%%%%%%%%%%%
%%%%%%%%%%%%%%%%%%%%%%%%%%%%%%%%%%%%%%%%%%%%%%
\begin{table}[htb]
\centering
\caption{BTFs from divisible difference sets}\label{tbl_dds}
\begin{tabular}{|P{.22\linewidth}|P{.40\linewidth}|P{.16\linewidth}| P{.10\linewidth}|}
\hline
 {\bf Suff. Conditions} &  $\pmb{ (n,m,l,\lambda,\mu)}$&   $\pmb{\alpha_{1},\alpha_{2}}$ %&\bf $\tau_{1},\tau_{2}$
& {\bf Ref.} \\
%%%%%%%%%%%%%%%%%%%%%%%%%%%%%%%%%%%%%
%%%%%%%%%%%%%%%%%%%%%%%%%%%%%%%%%% 

%%%%%%%%%%%%%%%%%%%%%%%%%
%%%%DDS ROW 1%%%%%%%%%%%%%%%%%%
%%%%%%%%%%%%%%%%%%%%%%%%%%%%%%
%%%%%%%%%%%%%%%%%%%%%%%%%%%%%%%
\hline
\begin{minipage}{0.22\columnwidth}%
\tiny
%%%%%%%%
%conditions%%%
%%%%%%%%
Suppose \ref{cond_mersenne}.  Let \\
$\scriptscriptstyle \delta=p^3+p^2-p-2$.
\end{minipage}
&
$\scriptscriptstyle
%%%%%%%%
%conditions%%%
%%%%%%%%
 \biggl( 
p^2(p+1), p(p+1), p^2, p, p+1
\biggr) 
$
&
$
\scriptscriptstyle
\tworows
{
     \alpha_{1}= 0
}
{
     \alpha_{2}= \frac{1}{p+1}
}
$
&
%$
%\scriptscriptstyle
%\tworows
%{
%     \tau_{1}= p + 1
%}
%{
%     \tau_{2}= \delta
%}
%$
%&
\begin{minipage}{0.1\columnwidth}%
\tiny
Prop~2.3 in \cite{MR1063380}.
\centering
%%%%%%%%%
%%notes%%%%
%%%%%%%%
\end{minipage}
\\

%%%%%%%%%%%%%%%%%%%%%%%%%
%%%%DDS ROW 2%%%%%%%%%%%%%%%%%%
%%%%%%%%%%%%%%%%%%%%%%%%%%%%%%
%%%%%%%%%%%%%%%%%%%%%%%%%%%%%%%
\hline
\begin{minipage}{0.22\columnwidth}%
\tiny
%%%%%%%%
%conditions%%%
%%%%%%%%
Suppose \ref{cond_mersenne}. Let \\
$\scriptscriptstyle \delta=p^3+p^2-p-2$.
\end{minipage}
&
$\scriptscriptstyle
%%%%%%%%
%conditions%%%
%%%%%%%%
 \biggl( 
p^2(p+1), p(2p-1), p^2, p(p-1),3(p-1)
\biggr) 
$
&
$
\scriptscriptstyle
\tworows
{
     \alpha_{1}= \frac{p-2}{2p-1}
}
{
     \alpha_{2}= \frac{1}{2p-1}
}
$
&
%$
%\scriptscriptstyle
%\tworows
%{
%     \tau_{1}= p+1
%}
%{
%     \tau_{2}= \delta
%}
%$
%&
\begin{minipage}{0.1\columnwidth}%
\tiny
Prop~2.9 in \cite{MR1063380}.
\centering
%%%%%%%%%
%%notes%%%%
%%%%%%%%
\end{minipage}
\\

%%%%%%%%%%%%%%%%%%%%%%%%%
%%%%DDS ROW 3%%%%%%%%%%%%%%%%%%
%%%%%%%%%%%%%%%%%%%%%%%%%%%%%%
%%%%%%%%%%%%%%%%%%%%%%%%%%%%%%%
\hline
\begin{minipage}{0.22\columnwidth}%
\tiny
%%%%%%%%
%conditions%%%
%%%%%%%%
Suppose $\scriptscriptstyle a \in \mathbb N$, where \\$\scriptscriptstyle a>1$ and $\scriptscriptstyle a$ is odd.
\end{minipage}
&
$\scriptscriptstyle
%%%%%%%%
%conditions%%%
%%%%%%%%
 \biggl( 
4a, a+2, a, a-2, 2
\biggr) 
$
&
$
\scriptscriptstyle
\tworows
{
     \alpha_{1}= \frac{a-2}{a+2}
}
{
     \alpha_{2}= \frac{2}{a+2}
}
$
&
%$
%\scriptscriptstyle
%\tworows
%{
%     \tau_{1}= 4
%}
%{
%     \tau_{2}= 4a-5
%}
%$
%&
\begin{minipage}{0.1\columnwidth}%
\tiny
Prop~2.12 in \cite{MR1063380}
\centering
%%%%%%%%%
%%notes%%%%
%%%%%%%%
\end{minipage}
\\

%%%%%%%%%%%%%%%%%%%%%%%%%
%%%%DDS ROW 4%%%%%%%%%%%%%%%%%%
%%%%%%%%%%%%%%%%%%%%%%%%%%%%%%
%%%%%%%%%%%%%%%%%%%%%%%%%%%%%%%
\hline
\begin{minipage}{0.22\columnwidth}%
\tiny
%%%%%%%%
%conditions%%%
%%%%%%%%
Suppose \ref{cond_primepower} and \\$q \equiv_4 1$.
\end{minipage}
&
$\scriptscriptstyle
%%%%%%%%
%conditions%%%
%%%%%%%%
 \biggl( 
2q,q,2, q-1, \frac{q-1}{2}
\biggr) 
$
&
$
\scriptscriptstyle
\tworows
{
     \alpha_{1}= \frac{1}{\sqrt{q}}
}
{
     \alpha_{2}= \frac 1 q
}
$
&
%$
%\scriptscriptstyle
%\tworows
%{
%     \tau_{1}=q
%}
%{
%     \tau_{2}= q-1
%}
%$
%&
\begin{minipage}{0.1\columnwidth}%
\tiny
\centering
Prop~2.13 in \cite{MR1063380}.
%%%%%%%%%
%%notes%%%%
%%%%%%%%
\end{minipage}
\\

%%%%%%%%%%%%%%%%%%%%%%%%%
%%%%DDS ROW 5%%%%%%%%%%%%%%%%%%
%%%%%%%%%%%%%%%%%%%%%%%%%%%%%%
%%%%%%%%%%%%%%%%%%%%%%%%%%%%%%%
\hline
\begin{minipage}{0.22\columnwidth}%
\tiny
%\centering
%%%%%%%%
%conditions%%%
%%%%%%%%
Suppose $\scriptscriptstyle a \in \mathbb N$.  Let
\\
$\scriptscriptstyle \delta = 3^{2a} - 2 \cdot 3^a$.
\end{minipage}
&
$\scriptscriptstyle
%%%%%%%%
%conditions%%%
%%%%%%%%
 \biggl( 
4 \cdot 3^{2a}, 
2\left(3^{2a} - 3^a\right), 
3^{2a},
\delta,
\delta+1
\biggr) 
$
&
$
\scriptscriptstyle
\tworows
{
     \alpha_{1}= 0
}
{
     \alpha_{2}= \frac{1}{2(3^a-1)}
}
$
&
%$
%\scriptscriptstyle
%\tworows
%{
%     \tau_{1}=4
%}
%{
%     \tau_{2}=4 \cdot 3^{2a} - 5
%}
%$
%&
\begin{minipage}{0.1\columnwidth}%
\tiny
\centering
%%%%%%%%%
%%notes%%%%
%%%%%%%%
Res.~2.3.9 in
\cite{MR1440858}.
\\
Also
\cite{MR1189879}.
\end{minipage}
\\

%%%%%%%%%%%%%%%%%%%%%%%%%%%%%%%
%%%%%%%DDS ROW 6%%%%%%%%%%%%%%%
%%%%%%%%%%%%%%%%%%%%%%%%%%%%%
%%%%%%%%%%%%%%%%%%%%%%%%%%%%%%%%
\hline
\begin{minipage}{0.22\columnwidth}%
\tiny
     Suppose \ref{cond_hadamard} and \ref{cond_diffset1}. Let 
$\scriptscriptstyle \delta= 2wu^2 +wu - 2uv$ and
$\scriptscriptstyle \epsilon = 2wu+w-2v.$
\end{minipage}
&
$\scriptscriptstyle
 \biggl( 4wu^2,\delta,   w,  \delta - 4u^2v +4u^2\left(\frac{v(v-1)}{w-1}\right) , \delta - wu^2 \biggr) 
$
&
$
\scriptscriptstyle
\tworows
{
     \alpha_{1}= \frac{|w-2v|}{\epsilon}
}
{
     \alpha_{2}=\frac{1}{\epsilon}\sqrt{\frac{4v(w-v)}{w-1}}
}
$
&
%$
%\scriptscriptstyle
%\tworows
%{
%     \tau_{1}=4u^2
%}
%{
%     \tau_{2}=4wu^2 -4u^2-1
%}
%$
%&
\begin{minipage}{0.1\columnwidth}%
\tiny
\centering
Cor.~2.3.2 in  \cite{MR1440858}. 
\\
Also \cite{MR658966}.
\end{minipage}
\\

%%%%%%%%%%%%%%%%%%%%%%%%%
%%%%DDS ROW 7%%%%%%%%%%%%%%%%%%
%%%%%%%%%%%%%%%%%%%%%%%%%%%%%%
%%%%%%%%%%%%%%%%%%%%%%%%%%%%%%%
\hline
\begin{minipage}{0.22\columnwidth}%
\tiny
%%%%%%%%
%conditions%%%
%%%%%%%%
Given
$\scriptscriptstyle a,b \in \mathbb N$, $\scriptscriptstyle a \leq b$, suppose \ref{cond_primepower} and \ref{cond_subgroup_EAq}.  Let\\
$\scriptscriptstyle \beta=q^{2b-a -1 }$,
$\scriptscriptstyle \delta=\frac{q^{a-1} -1}{q-1}$ and
$\scriptscriptstyle \epsilon=\frac{q^{a} -1}{q-1}$.
\end{minipage}
&
$\scriptscriptstyle
%%%%%%%%
%conditions%%%
%%%%%%%%
 \biggl( 
\epsilon q^{2b-a},
\epsilon \beta,
q^a,
\delta \beta,
\epsilon \beta q^{-1}
\biggr) 
$
&
$
\scriptscriptstyle
\tworows
{
     \alpha_{1}=0
}
{
     \alpha_{2}=\frac{q^{a-b}}{\epsilon}
}
$
&
%$
%\scriptscriptstyle
%\tworows
%{
%     \tau_{1}=\epsilon q^{2(b-a)}
%}
%{
%     \tau_{2}= \epsilon q^{2(b-a)}(q^a-1)-1
%}
%$
%&
\begin{minipage}{0.1\columnwidth}%
\tiny
\centering
%%%%%%%%%
%%notes%%%%
%%%%%%%%
Thm.~2.3.6 in
\cite{MR1440858}.
\end{minipage}
\\

\hline

\end{tabular}

\end{table}

%\newline
%\vspace*{1 cm}
%\newline

Relative difference sets are a class of divisible difference sets which have received special attention in the literature  \cite{MR1400404,
GodsilRoy2009, MR3557826}.  For this reason, we tabulate examples of biangular harmonic frames generated by relative difference sets separately from the general divisible difference sets.  As with the divisible difference sets, we point out that our notation for relative difference sets is not standard; see Remark~\ref{rem_DDS_RDS_notn}.

%%%%%%%%%%%%%%%%%%%%%%%%%%%%%%%%%%%%%%%%%%%%%%%%%%%%%%%%%%%%%%%%%%%%%%%%%%%%%%%%%%%%%%%%%%%%%%
%%%%%%%%%RDS TABLE%%%%%%%%%%%%%%%%%%%%%%%%%%%%%%%%
%%%%%%%%%%%%%%%%%%%%%%%%%%%%%%%%%%%%%%%%%%%%%%
%%%%%%%%%%%%%%%%%%%%%%%%%%%%%%%%%%%%%%%%%%%%%%

\begin{table}[htb]
\caption{BTFs from relative difference sets}\label{tbl_rds}
\begin{tabular}{|P{.22\linewidth}|P{.40\linewidth}|P{.16\linewidth}| P{.10\linewidth}|}
\hline
 {\bf Suff. Conditions} &  $\pmb{(n,m,l,\mu)}$&   $\pmb{\alpha_{1},\alpha_{2}}$ %&\bf $\tau_{1},\tau_{2}$
&  {\bf Ref.} \\
%%%%%%%%%%%%%%%%%%%%%%%%%%%%%%%%%%%%%
%%%%%%%%%%%%%%%%%%%%%%%%%%%%%%%%%% 

%%%%%%%%%%%%%%%%%%%%%%%%%
%%%%RDS ROW1 %%%%%%%%%%%%%%%%%%
%%%%%%%%%%%%%%%%%%%%%%%%%%%%%%
%%%%%%%%%%%%%%%%%%%%%%%%%%%%%%%
\hline
\begin{minipage}{0.22\columnwidth}%
\tiny
%%%%%%%%
%conditions%%%
%%%%%%%%
Suppose $\scriptscriptstyle a, b \in \mathbb N$, $\scriptscriptstyle a \leq b$ and \ref{cond_prime}.
\end{minipage}
&
$\scriptscriptstyle
%%%%%%%%
%conditions%%%
%%%%%%%%
 \biggl( 
p^{a+b}, p^b, p^a, p^{b-a}
\biggr) 
$
&
$
\scriptscriptstyle
\tworows
{
     \alpha_{1}=0
}
{
     \alpha_{2}= p^{-b/2}
}
$
&
%$
%\scriptscriptstyle
%\tworows
%{
%     \tau_{1}=
%}
%{
%     \tau_{2}=
%}
%$
%&
\begin{minipage}{0.1\columnwidth}%
\tiny
Sect~3.1 in \cite{MR1400404}.
\centering
%%%%%%%%%
%%notes%%%%
%%%%%%%%
\end{minipage}
\\

%%%%%%%%%%%%%%%%%%%%%%%%%
%%%%RDS ROW 2%%%%%%%%%%%%%%%%%%
%%%%%%%%%%%%%%%%%%%%%%%%%%%%%%
%%%%%%%%%%%%%%%%%%%%%%%%%%%%%%%
\hline
\begin{minipage}{0.22\columnwidth}%
\tiny
%%%%%%%%
%conditions%%%
%%%%%%%%
Suppose \ref{cond_hadamard}.
\end{minipage}
&
$\scriptscriptstyle
%%%%%%%%
%conditions%%%
%%%%%%%%
 \biggl( 
8u^2,
4u^2,
2,
2u^2
\biggr) 
$
&
$
\scriptscriptstyle
\tworows
{
     \alpha_{1}=0
}
{
     \alpha_{2}=\frac{1}{2u}
}
$
&
%$
%\scriptscriptstyle
%\tworows
%{
%     \tau_{1}=
%}
%{
%     \tau_{2}=
%}
%$
%&
\begin{minipage}{0.1\columnwidth}%
\tiny
Sect~3.1 in \cite{MR1400404}.
\centering
%%%%%%%%%
%%notes%%%%
%%%%%%%%
\end{minipage}
\\

%%%%%%%%%%%%%%%%%%%%%%%%%
%%%%RDS ROW 3%%%%%%%%%%%%%%%%%%
%%%%%%%%%%%%%%%%%%%%%%%%%%%%%%
%%%%%%%%%%%%%%%%%%%%%%%%%%%%%%%
\hline
\begin{minipage}{0.22\columnwidth}%
\tiny
%%%%%%%%
%conditions%%%
%%%%%%%%
Suppose \ref{cond_hadamard}.
\end{minipage}
&
$\scriptscriptstyle
%%%%%%%%
%conditions%%%
%%%%%%%%
 \biggl(
16u^2, 8u^2, 2, 4u^2 
\biggr) 
$
&
$
\scriptscriptstyle
\tworows
{
     \alpha_{1}=0
}
{
     \alpha_{2}=\frac{\sqrt{2}}{4u}
}
$
&
%$
%\scriptscriptstyle
%\tworows
%{
%     \tau_{1}=
%}
%{
%     \tau_{2}=
%}
%$
%&
\begin{minipage}{0.1\columnwidth}%
\tiny
Sect~3.1 in \cite{MR1400404}.
\centering
%%%%%%%%%
%%notes%%%%
%%%%%%%%
\end{minipage}
\\

%%%%%%%%%%%%%%%%%%%%%%%%%
%%%%RDS ROW 4%%%%%%%%%%%%%%%%%%
%%%%%%%%%%%%%%%%%%%%%%%%%%%%%%
%%%%%%%%%%%%%%%%%%%%%%%%%%%%%%%
\hline
\begin{minipage}{0.22\columnwidth}%
\tiny
%%%%%%%%
%conditions%%%
%%%%%%%%
Given \ref{cond_primepower} and $\scriptscriptstyle a \in \mathbb N$, let $\scriptscriptstyle d\in \mathbb N$ such that $\scriptscriptstyle d | q-1$.
\end{minipage}
&
$\scriptscriptstyle
%%%%%%%%
%conditions%%%
%%%%%%%%
 \biggl( 
\frac{q^{a+1} -1}{d}, q^{a}, \frac{q-1}{d}, dq^{a-1} 
\biggr) 
$
&
$
\scriptscriptstyle
\tworows
{
     \alpha_{1}=q^{-(a+1)/2}
}
{
     \alpha_{2}=q^{-a/2}
}
$
&
%$
%\scriptscriptstyle
%\tworows
%{
%     \tau_{1}=
%}
%{
%     \tau_{2}=
%}
%$
%&
\begin{minipage}{0.1\columnwidth}%
\tiny
Sect~3.3 in \cite{MR1400404}.
\centering
%%%%%%%%%
%%notes%%%%
%%%%%%%%
\end{minipage}
\\

%%%%%%%%%%%%%%%%%%%%%%%%%
%%%%RDS ROW 5%%%%%%%%%%%%%%%%%%
%%%%%%%%%%%%%%%%%%%%%%%%%%%%%%
%%%%%%%%%%%%%%%%%%%%%%%%%%%%%%%
\hline
\begin{minipage}{0.22\columnwidth}%
\tiny
%%%%%%%%
%conditions%%%
%%%%%%%%
Given \ref{cond_primepower} and $\scriptscriptstyle a \in \mathbb N$, where $\scriptscriptstyle q$ and $\scriptscriptstyle a$ are both even, let $\scriptscriptstyle \delta = \frac{q-1}{2}$.
\end{minipage}
&
$\scriptscriptstyle
%%%%%%%%
%conditions%%%
%%%%%%%%
 \biggl( 
\frac{(q^{a+1} -1)}{\delta}, q^{a}, 2, \delta q^{a-1}
\biggr) 
$
&
$
\scriptscriptstyle
\tworows
{
     \alpha_{1}=q^{-(a+1)/2}
}
{
     \alpha_{2}=q^{-a/2}
}
$
&
%$
%\scriptscriptstyle
%\tworows
%{
%     \tau_{1}=
%}
%{
%     \tau_{2}=
%}
%$
%&
\begin{minipage}{0.1\columnwidth}%
\tiny
Sect~3.3 in \cite{MR1400404}.
%Also \cite{}.
\centering
%%%%%%%%%
%%notes%%%%
%%%%%%%%
\end{minipage}
\\

\hline

\end{tabular}

\end{table}

\begin{remark}
Adjoining the canonical orthonormal basis to a biangular harmonic frame generated by a relative difference set from the first row of Table~\ref{tbl_rds}, where we take $a=b$, yields a {\it maximal set of mutually unbiased bases} in $\mathbb C^{p^a}$.  See \cite{GodsilRoy2009} for details.
\end{remark}

\begin{remark}
Adjoining the canonical orthonormal basis to a biangular harmonic frame generated by a relative difference set from the fourth row of Table~\ref{tbl_rds}, where we take $d=1$, yields a so-called {\it orthoplectic Grassmannian frame}, which generates an optimal line packing of $q^{a+1}+q^a-1$ lines in $\mathbb C^{q^a}$.  See \cite{MR3557826} for details.
\end{remark}

In order to use Theorem~\ref{cor_pds_btf} to compute the frame angles  of the BTFs generated by the  partial difference sets in Table~\ref{tbl_pds}, we remark that all of the partial difference sets we have collected here are regular; see Section~\ref{sec_pds}.

\begin{rem}
By Corollary~\ref{cor_pds_btf_pairs}, it follows that for each biangular harmonic frame $\mathcal F$ for $\mathbb C^m$ produced by a regular partial difference set $\mathcal S$ from Table~\ref{tbl_pds}, 
there exists an equiangular or biangular harmonic frame $\mathcal F'$ for $\mathbb C^{m+1}$ generated by $\mathcal S \cup \{0_G\}$.  The frame angles of $\mathcal F'$ can be calculated with Theorem~\ref{cor_pds_btf}.
\end{rem}

\begin{rem}
If we take $v=1$ and $b=p_1^{a_1}+1$ in the fourth row of Table~\ref{tbl_pds}, then the corresponding partial difference set is a difference set, the claimed value for $\mu$ is vacuously true and $\alpha_2$ can be disregarded.  The equiangular harmonic frame produced in this special case  corresponds to a simplex.
\end{rem}

%%%%%%%%%%%%%%%%%%%%%%%%%%%%%%%%%%%%%%%%%%%%%%%%%%%%%%%%%%%%%%%%%%%%%%%%%%%%%%%%%%%%%%%%%%%%%%%%%%%%%%%%%%%%%%%%%%%%%%%%%%%%%%%%%%%%%%%%%%%
%%%%%%%%%%%%%PDS TABLE%%%%%%%%%%%%%%%%%%%%%%%%%%%
%%%%%%%%%%%%%%%%%%%%%%%%%%%%%%%%%%%%%%%%%%%%%
%%%%%%%%%%%%%%%%%%%%%%%%%%%%%%%%%%%%%%%%%%%%%

\begin{table}[t]
\caption{BTFs from regular partial difference sets}\label{tbl_pds}
\begin{tabular}{|P{.22\linewidth}|P{.40\linewidth}|P{.16\linewidth}| P{.10\linewidth}|}
\hline
 {\bf Suff. Conditions} &  $\pmb{(n,m,\lambda,\mu)}$ &   $\pmb{\alpha_{1},\alpha_{2}}$ &  {\bf Ref.} \\
%%%%%%%%%%%%%%%%%%%%%%%%%%%%%%%%%%%%%
%%%%%%%%%%%%%%%%%%%%%%%%%%%%%%%%%% 

%%%%%%%%%%%%%%%%%%%%%%%%%
%%%%ROW 1%%%%%%%%%%%%%%%%%%
%%%%%%%%%%%%%%%%%%%%%%%%%%%%%%
%%%%%%%%%%%%%%%%%%%%%%%%%%%%%%%
\hline
\begin{minipage}{0.22\columnwidth}%
\tiny
%%%%%%%%
%conditions%%%
%%%%%%%%
Suppose \ref{cond_primepower} and $\scriptscriptstyle q \equiv_4 1$.
\end{minipage}
&
$\scriptscriptstyle
%%%%%%%%
%conditions%%%
%%%%%%%%
 \biggl( 
q, \frac{q-1}{2}, \frac{q-5}{4}, \frac{q-1}{4}
\biggr) 
$
&
$
\scriptscriptstyle
\tworows
{
     \alpha_{1}= \frac{1}{\sqrt{q} +1}
}
{
     \alpha_{2}= \frac{1}{\sqrt{q} -1}
}
$
&
%$
%\scriptscriptstyle
%\tworows
%{
%     \tau_{1}=
%}
%{
%     \tau_{2}=
%}
%$
%&
\begin{minipage}{0.1\columnwidth}%
\tiny
Thm~2.1 in \cite{MR1277942}.
\centering
%%%%%%%%%
%%notes%%%%
%%%%%%%%
\end{minipage}
\\

%%%%%%%%%%%%%%%%%%%%%%%%%
%%%%ROW 2%%%%%%%%%%%%%%%%%%
%%%%%%%%%%%%%%%%%%%%%%%%%%%%%%
%%%%%%%%%%%%%%%%%%%%%%%%%%%%%%%
\hline
\begin{minipage}{0.22\columnwidth}%
\tiny
%%%%%%%%
%conditions%%%
%%%%%%%%
Let $\scriptscriptstyle a \in \mathbb N$, where $\scriptscriptstyle a>1$.
\end{minipage}
&
$\scriptscriptstyle
%%%%%%%%
%conditions%%%
%%%%%%%%
 \biggl( 
a^2, 2(a-1), a-2, 2
\biggr) 
$
&
$
\scriptscriptstyle
\tworows
{
     \alpha_{1}= \frac{a-2}{2(a-1)}
}
{
     \alpha_{2}= \frac{1}{a-1}
}
$
&
%$
%\scriptscriptstyle
%\tworows
%{
%     \tau_{1}=
%}
%{
%     \tau_{2}=
%}
%$
%&
\begin{minipage}{0.1\columnwidth}%
\tiny
Ex~2.3.1 in \cite{MR1277942}.
\centering
%%%%%%%%%
%%notes%%%%
%%%%%%%%
\end{minipage}
\\

%%%%%%%%%%%%%%%%%%%%%%%%%
%%%%ROW 3%%%%%%%%%%%%%%%%%%
%%%%%%%%%%%%%%%%%%%%%%%%%%%%%%
%%%%%%%%%%%%%%%%%%%%%%%%%%%%%%%
\hline
\begin{minipage}{0.22\columnwidth}%
\tiny
%%%%%%%%
%conditions%%%
%%%%%%%%
Let $\scriptscriptstyle a \in \mathbb N$, where $\scriptscriptstyle a>1$.
\end{minipage}
&
$\scriptscriptstyle
%%%%%%%%
%conditions%%%
%%%%%%%%
 \biggl( 
a^2, 3(a-1), a, 6
\biggr) 
$
&
$
\scriptscriptstyle
\tworows
{
     \alpha_{1}= \frac{a-3}{3(a-1)}
}
{
     \alpha_{2}= \frac{1}{a-1}
}
$
&
%$
%\scriptscriptstyle
%\tworows
%{
%     \tau_{1}=
%}
%{
%     \tau_{2}=
%}
%$
%&
\begin{minipage}{0.1\columnwidth}%
\tiny
Ex~2.3.1 in \cite{MR1277942}.
\centering
%%%%%%%%%
%%notes%%%%
%%%%%%%%
\end{minipage}
\\

%%%%%%%%%%%%%%%%%%%%%%%%%
%%%%ROW 4%%%%%%%%%%%%%%%%%%
%%%%%%%%%%%%%%%%%%%%%%%%%%%%%%
%%%%%%%%%%%%%%%%%%%%%%%%%%%%%%%
\hline
\begin{minipage}{0.22\columnwidth}%
\tiny
%%%%%%%%
%conditions%%%
%%%%%%%%
Given \ref{cond_prime_decomp}, let\\
$\scriptscriptstyle
c = p_1^{a_1} ... p_v^{a_v}$ and $\scriptscriptstyle b \in \mathbb N$,
 where
$
\scriptscriptstyle b\, \leq \, \min\limits_{j} \left\{p_j^{a_j}+1 \right\}.
$
\end{minipage}
&
$\scriptscriptstyle
%%%%%%%%
%conditions%%%
%%%%%%%%
 \biggl( 
c^2, b(c-1), c+b^2-3b, b^2-b
\biggr) 
$
&
$
\scriptscriptstyle
\tworows
{
     \alpha_{1}= \frac{|c-b|}{b(c-1)}
}
{
     \alpha_{2}=  \frac{1}{c-1}
}
$
&
%$
%\scriptscriptstyle
%\tworows
%{
%     \tau_{1}=
%}
%{
%     \tau_{2}=
%}
%$
%&
\begin{minipage}{0.1\columnwidth}%
\tiny
Cor~2.5 in \cite{MR1277942}.
\centering
%%%%%%%%%
%%notes%%%%
%%%%%%%%
\end{minipage}
\\

%%%%%%%%%%%%%%%%%%%%%%%%%
%%%%ROW 5%%%%%%%%%%%%%%%%%%
%%%%%%%%%%%%%%%%%%%%%%%%%%%%%%
%%%%%%%%%%%%%%%%%%%%%%%%%%%%%%%
\hline
\begin{minipage}{0.22\columnwidth}%
\tiny
%%%%%%%%
%conditions%%%
%%%%%%%%
Suppose $\scriptscriptstyle a \in \mathbb N$ and \ref{cond_prime} with $\scriptscriptstyle p$ odd. Let\\
$
\scriptscriptstyle
\delta = 9 \cdot p^{4a}.$
%$
% and 
%$\scriptscriptstyle \epsilon= \frac{\delta +1}{2}$.
\end{minipage}
&
$\scriptscriptstyle
%%%%%%%%
%conditions%%%
%%%%%%%%
 \biggl( 
\delta, \frac{\delta -1}{2},  \frac{\delta -5}{4},  \frac{\delta -1}{4}
\biggr) 
$
&
$
\scriptscriptstyle
\tworows
{
     \alpha_{1}= \frac{1}{\sqrt{\delta}+1}
}
{
     \alpha_{2}= \frac{1}{\sqrt{\delta}-1}
}
$
&
%$
%\scriptscriptstyle
%\tworows
%{
%     \tau_{1}=
%}
%{
%     \tau_{2}=
%}
%$
%&
\begin{minipage}{0.1\columnwidth}%
\tiny
Thm~3.1 in \cite{MR2496253}.
\centering
%%%%%%%%%
%%notes%%%%
%%%%%%%%
\end{minipage}
\\

%%%%%%%%%%%%%%%%%%%%%%%%%
%%%%ROW 6%%%%%%%%%%%%%%%%%%
%%%%%%%%%%%%%%%%%%%%%%%%%%%%%%
%%%%%%%%%%%%%%%%%%%%%%%%%%%%%%%
\hline
\begin{minipage}{0.22\columnwidth}%
\tiny
%%%%%%%%
%conditions%%%
%%%%%%%%
Suppose $\scriptscriptstyle a \in \mathbb N$ and \ref{cond_prime} with $\scriptscriptstyle p$ odd.  Let\\
$
\scriptscriptstyle
\delta = 3 p^{2a}
$ and let 
$\scriptscriptstyle
\epsilon= \frac{\delta - 3}{2}
.$
\end{minipage}
&
$\scriptscriptstyle
%%%%%%%%
%conditions%%%
%%%%%%%%
 \biggl( 
\delta^2, 
\epsilon (\delta +1), 
-\delta + \epsilon^2 + 3\epsilon,  
\epsilon^2 +\epsilon
\biggr) 
$
&
$
\scriptscriptstyle
\tworows
{
     \alpha_{1}= \frac{1}{\delta + 1}
}
{
     \alpha_{2}= \frac{\delta-\epsilon}{\epsilon(\delta +1)}
}
$
&
%$
%\scriptscriptstyle
%\tworows
%{
%     \tau_{1}=
%}
%{
%     \tau_{2}=
%}
%$
%&
\begin{minipage}{0.1\columnwidth}%
\tiny
Thm~3.2 in \cite{MR2496253}.
\centering
%%%%%%%%%
%%notes%%%%
%%%%%%%%
\end{minipage}
\\

%%%%%%%%%%%%%%%%%%%%%%%%%
%%%%ROW 7%%%%%%%%%%%%%%%%%%
%%%%%%%%%%%%%%%%%%%%%%%%%%%%%%
%%%%%%%%%%%%%%%%%%%%%%%%%%%%%%%
\hline
\begin{minipage}{0.22\columnwidth}%
\tiny
%%%%%%%%
%conditions%%%
%%%%%%%%
Suppose $\scriptscriptstyle a \in \mathbb N$.
Let \\$\scriptscriptstyle \beta = 2^{2a-1} - 2^{a-1}$,\\
$\scriptscriptstyle \delta =  2^{a-1}-1$
and
$\scriptscriptstyle \epsilon =  2^{a}-1$.
\end{minipage}
&
$\scriptscriptstyle
%%%%%%%%
%conditions%%%
%%%%%%%%
 \biggl( 
2^{3a}, 
\beta\epsilon,
2^{a-1} + \beta(\delta-1),\beta\delta
\biggr) 
$
&
$
\scriptscriptstyle
\tworows
{
     \alpha_{1}= \epsilon^{-2}
}
{
     \alpha_{2}= \epsilon^{-1}
}
$
&
%$
%\scriptscriptstyle
%\tworows
%{
%     \tau_{1}=
%}
%{
%     \tau_{2}=
%}
%$
%&
\begin{minipage}{0.1\columnwidth}%
\tiny
Thm~3.2 in \cite{MR1797679}.
\centering
%%%%%%%%%
%%notes%%%%
%%%%%%%%
\end{minipage}
\\

%%%%%%%%%%%%%%%%%%%%%%%%%
%%%%ROW 8%%%%%%%%%%%%%%%%%%
%%%%%%%%%%%%%%%%%%%%%%%%%%%%%%
%%%%%%%%%%%%%%%%%%%%%%%%%%%%%%%
\hline
\begin{minipage}{0.22\columnwidth}%
\tiny
%%%%%%%%
%conditions%%%
%%%%%%%%
Suppose $\scriptscriptstyle a\in \mathbb N$, where \\$\scriptscriptstyle a>1$ and $\scriptscriptstyle a$ is odd. Let\\
$\scriptscriptstyle \delta = 4^{a-1}-1$ and $\scriptscriptstyle \epsilon = 4^{a-1}$.
\end{minipage}
&
$\scriptscriptstyle
%%%%%%%%
%conditions%%%
%%%%%%%%
 \biggl( 
4^{2a},
(4^a + 1)\delta,
\epsilon^2 - 3\epsilon -2,
\delta\epsilon
\biggr) 
$
&
$
\scriptscriptstyle
\tworows
{
     \alpha_{1}=(4^a+1)^{-1}
}
{
     \alpha_{2}=\frac{3\epsilon +1}{\delta(4^a+1)}
}
$
&
%$
%\scriptscriptstyle
%\tworows
%{
%     \tau_{1}=
%}
%{
%     \tau_{2}=
%}
%$
%&
\begin{minipage}{0.1\columnwidth}%
\tiny
Cor~2.2 in \cite{MR2064754}.
\centering
%%%%%%%%%
%%notes%%%%
%%%%%%%%
\end{minipage}
\\

\hline

\end{tabular}

\end{table}

%
%\afterpage{\clearpage}
%%\FloatBarrier
%%\section*{}
%\clearpage
\bibliography{harmbtfsbib}
\bibliographystyle{plain}

\end{document}